\documentclass[12pt]{article}
\usepackage{amsmath,amssymb,mathtools}
\usepackage{amsthm}
\usepackage{amscd}
\usepackage{epsfig}
\usepackage{graphics}
\usepackage{graphicx}
\usepackage{colordvi}
\usepackage{enumitem}
\usepackage{mathrsfs} 
\usepackage[colorlinks]{hyperref}
\usepackage{dsfont}
\usepackage{mparhack}
\usepackage{xcolor,lipsum}
\usepackage[displaymath, mathlines]{lineno}
\modulolinenumbers[5]
\nolinenumbers

\hypersetup{
breaklinks,
colorlinks,
linkcolor=blue!75!black,
urlcolor=blue!75!black,
citecolor=blue!75!black}

\theoremstyle{plain}

\theoremstyle{definition}

\theoremstyle{remark}

\numberwithin{equation}{section}

\newcommand{\1}{\mathds 1}

\newcommand{\C}{{\mathscr C}}
\newcommand{\F}{\mathscr F}

\newcommand{\R}{{\mathbb{R}}}

\newcommand{\scrP}{{\mathscr P}}

\newcommand{\vep}{{\varepsilon}}

\renewcommand{\to}{\longrightarrow}

\newcommand{\lmt}{\longmapsto}

\newcommand{\ind}{{\perp\!\!\!\perp}}

\newcommand{\ms}{\mathscr}
\renewcommand{\P}{{\mathbb P}}

\newcommand{\E}{{\mathbb  E}}
\newcommand{\bfP}{{\P}}

\newcommand{\En}{\mathbb{E}^{(n)}}

\newcommand{\hphi}{{\phi\circ\mathbf m}}

\newcommand{\pid}{\pi_{\rm diag}}

\newcommand{\FV}{{\sf FV}}

\newcommand{\xin}{\xi^{(n)}}

\newcommand{\convdn}{\xrightarrow[n\to\infty]{{\rm (d)}}}

\newcommand{\pinmax}{\pi^{(n)}_{\rm max}}

\newcommand{\tmixn}{\mathbf t_{\rm mix}^{(n)}}

\newcommand{\tmeetn}{\gamma_n}

\newcommand{\To}{\xrightarrow[n\to\infty]{{\rm (d)}}}

\newcommand{\VM}{{\sf VM}}

\oddsidemargin
.3in \evensidemargin 0.1in\marginparwidth 0.4in

\begin{document}

\newtheoremstyle{slantthm}{10pt}{10pt}{\slshape}{}{\bfseries}{}{.5em}{\thmname{#1}\thmnumber{ #2}\thmnote{ (#3)}.}
\newtheoremstyle{slantthmp}{10pt}{10pt}{\slshape}{}{\bfseries}{}{.5em}{\thmname{#1}\thmnumber{ #2}\thmnote{ (#3)}.}
\newtheoremstyle{slantrmk}{10pt}{10pt}{\rmfamily}{}{\bfseries}{}{.5em}{\thmname{#1}\thmnumber{ #2}\thmnote{ (#3)}.}

\theoremstyle{slantthm}
\newtheorem{thm}{Theorem}[section]
\newtheorem{prop}[thm]{Proposition}
\newtheorem{lem}[thm]{Lemma}
\newtheorem{cor}[thm]{Corollary}
\newtheorem{defi}[thm]{Definition}
\newtheorem{prob}[thm]{Problem}
\newtheorem{disc}[thm]{Discussion}
\newtheorem*{nota}{Notation}
\newtheorem*{conj}{Conjecture}
\theoremstyle{slantrmk}
\newtheorem{rmk}[thm]{Remark}
\newtheorem{eg}[thm]{Example}
\newtheorem{step}{Step}
\newtheorem{claim}{Claim}
\theoremstyle{plain}
\newtheorem{thmm}{Theorem}[section]

\title{{\bf Weak atomic convergence of finite voter models toward Fleming-Viot processes}}

\author{Yu-Ting Chen\footnote{Department of Mathematics, University of Tennessee, Knoxville, TN, US}~ and J. Theodore Cox\footnote{Department of Mathematics, Syracuse University, Syracuse, NY, US}}

\maketitle

\begin{abstract}
We consider the empirical measures of multi-type voter
models with mutation on large finite sets, and prove their weak
atomic convergence in the sense of Ethier and Kurtz (1994)  toward a Fleming-Viot process.
Convergence in the weak atomic topology is
strong enough to answer a line of inquiry raised by Aldous (2013)
 concerning the distributions of the 
 corresponding entropy processes and diversity processes for types.
\end{abstract}

\tableofcontents

\section{Introduction}\label{sec:intro}
Voter models are a generalization of Moran processes \cite{Moran_1958} from population genetics  in the presence of spatial structure, and have been one of the major subjects in interacting particle systems  \cite{L:IPS} along with their variants in models of cancer, social dynamics, and probabilistic evolutionary games. See, for example, \cite{Aldous_FMIE, Bramson_1980, Bramson_1981, Chen_2013, CDP_13}, and the references there for origins of these models. 
The present paper is a continuation of Chen, Choi and
Cox~\cite{CCC}
which considers the weak convergence of two-type voter models toward the Wright-Fisher diffusion.  
Our main goal here is to show that with respect to the weak
atomic convergence introduced by Ethier and Kurtz
\cite{EK:AT}, which is finer than the usual weak
convergence, infinite-type voter models on large finite sets
in the presence of mutation converge to a Fleming-Viot process. Fleming-Viot processes are one
of the most well studied classes of measure-valued processes, in part
due to its duality with the Kingman coalescent (see
\cite{Birkner_2009, Donnelly_1996, Donnelly_1999, EK:93} and
many others).

For the voter models considered throughout this paper
we allow
multiple types and mutation. The models are defined as
follows. With respect to a finite set $E$ with size
$\#E=N\geq 2$ and  a compact metric
type space $S$, a multi-type voter model is defined by a ``voting mechanism'' according to an irreducible transition kernel $q$ on $E$ with zero trace $q(x,x)\equiv 0$,
and incorporates independent mutation according to a finite
measure $\mu$ on $S$. The kernel
$q$ plays the role of a ``voting kernel'' in that at rate 1, the type at each $x\in E$
is replaced by the type at $y$ with probability $q(x,y)$. On
the other hand, if $\mu$ is nonzero, mutation at each $x\in E $ occurs independently with rate $\mu(\1)$, and a new type is chosen according to $\mu(\,\cdot\,)/\mu(\1)$. 
The canonical examples for these voter models are defined by
voting kernels which are random walk transition probabilities on finite, connected, edge-weighted graphs without self-loops. Here and in what follows, see \cite{Biggs_1974} for terminology in graph theory.

Voter models where voting kernels are random walk transition probabilities on complete graphs reduce to the classical Moran processes. This particular case forms the basis of several important probability models. For example, in the limit of large $N$, frequencies of a fixed type under two-type Moran processes converge to the Wright-Fisher diffusion. 
Furthermore, in the multi-type case, the empirical
   measures of the corresponding Moran processes converge to a Fleming-Viot process,
   which is a measure-valued 
   infinite dimensional generalization of the Wright-Fisher
   diffusion (cf. \cite[Chapter~10]{EK:MP}, and also
Section~\ref{sec:coal} below).
Along these lines, one of the major interests has been in the rich properties of the continuum limits, whereas the mean-field nature of Moran processes allows complete characterizations of those dynamical equations on their own.

The objects of this paper are the empirical measures of voter models on large finite spatial structures. The empirical measure associated with an $(E,q,\mu)$-voter model $(\xi_t)$ is given by the process $\big(\mathbf m(\xi_t)\big)$ taking values in the space $\ms P(S)$ of probability measures on $S$. Here, the probability-measure-valued functional $\mathbf m$ is defined by
\begin{equation}\label{eq:empmeas}
\mathbf m(\xi)\doteq \sum_{x\in E}\pi(x)\delta_{\xi(x)},\quad \xi\in S^E,
\end{equation}
where the weight function $\pi$ is  the unique stationary distribution associated with the voting kernel $q$. Notice that the mass of $\mathbf m(\xi)$ at $\sigma$, for $\sigma\in S$, gives the ($\pi$-weighted) density of type $\sigma$ in $\xi$. 

The first main result of this paper,
  Theorem~\ref{thmm:1} below, generalizes the classical result of convergence of multi-type Moran processes to a Fleming-Viot process, 
  and is an infinite-dimensional generalization of \cite[Theorem~2.2]{CCC} where two-type voter models without mutation are
investigated. With respect to an appropriate sequence of
$(E_n,q^{(n)},\mu_n)$-voter models $(\xi^{(n)}_t)$ where $\#
E_n$ increases to infinity,
we establish the weak convergence 
\begin{align}\label{mN}
\mathbf m\big(\xi^{(n)}_{\gamma_n \cdot}\big) \xrightarrow[n\to\infty]{} X
\end{align}
as probability-measure-valued processes on the type space $S$ (with respect to Skorokhod's $J_1$-topology), where
$X$ is a Fleming-Viot process.  Here, the
time scales $\gamma_n$  are growing constants given by the
expected first meeting times of two independent Markov
chains which are driven by the corresponding voting kernels
and are started from stationarity, and
$\ms P(S)$ is equipped with the usual topology of weak
convergence. The major conditions for the weak convergence in (\ref{mN}) are certain
simple, mild mixing conditions on the
kernels $q^{(n)}$ expressed in terms of total variation mixing times,
or spectral gaps when the kernels are
reversible; we also require ``weak mutation'' so that $\gamma_n\mu_n$
converges weakly. Either of our
two mixing conditions (see Theorem~\ref{thmm:1}(iv)) guarantees that two independent~$q^{(n)}$-Markov
chains get appropriately close
to stationarity before they meet. They are particular
formalizations of the general principle appearing as early
as in Keilson
\cite{Keilson_1979} and Aldous \cite{Aldous_1982} to obtain
the convergence of rescaled hitting times toward exponential
variables. See Oliveira~\cite{Oliveira_2012, Oliveira_2013}
for a closely related application of this
  principle to absorption of voter models without mutation, and Section~\ref{sec:atom} for an application of Oliveira's condition to weak atomic convergence of voter models (to be discussed in more detail below).

The proof
of (\ref{mN}) makes
  use of the standard duality between voter models and
coalescing Markov chains driven by voting kernels (see
Section~\ref{sec:coal}). The mixing conditions mentioned above applied to the dual processes lead to the mean-field-like behavior for  dynamics of the empirical measures of voter models through duality. In this direction, our application of duality will
be kept at a minimum in order to obtain mild conditions on voting kernels for the weak convergence. In more detail, the  well-posedness of the martingale problem satisfied by the limiting Fleming-Viot process only requires us to characterize
the limits of the first two moments of the empirical measures (against appropriate test functions). On the other hand, the use of the empirical measures in (\ref{eq:empmeas}) means that 
we consider a direction different from the method of
stochastic PDEs for rescaled (two-type) voter models on integer
lattices by Mueller and Tribe~\cite{MT} and Cox, Durrett and Perkins~\cite{CDP}. With respect to a fixed type $\sigma$, the stochastic PDE method considers certain rescaled limits of the discrete-measure-valued process $\sum_x \1_{\{\sigma\}}\big(\xi_t(x)\big)\delta_x$ for a voter model $(\xi_t)$ defined on an integer lattice, where $x$ ranges over all vertices. 
In contrast to the context in this paper,
the special geometry of integer lattices is required for this method. Nonetheless, as far as a fixed type is concerned, the method can give a much more detailed description
of the space-time dynamics of type $\sigma$ under the voter model, in terms of  processes with ``cleaner'' characterizations.

Up to this point, our program for the weak convergence of
 empirical measures for multi-type voter models is similar to that in
\cite{CCC} for the weak convergence of density processes for two-type voter models without mutation toward the Wright-Fisher diffusion. However,  the multi-type
  character of the voter models under consideration and the presence of mutation
now bring new technical
  issues. In particular, the latter changes the delicate duality relation between the voter models and the coalescing Markov chains.
The first part of this paper will be devoted to the resolutions of these issues. See Section~\ref{sec:emvm} and Section~\ref{sec:coal} for details.

The second main result of this paper strengthens the mode of
the convergence in \eqref{mN}. Our motivation in this direction is a line of inquiry by Aldous~\cite{Aldous_FMIE} on ``finite Markov information exchange models''. 
The context there considers social
dynamics for a large population of agents which are located
at
the vertices of finite, connected,
weighted graphs without self-loops. Their ``opinions''
evolve according to voter model dynamics $(\xi_t)$ without
mutation. Different opinions are held by the agents in
the initial states so that the number of effective types in the
system grows with the size of state space.

The inquiry in \cite[Section 5.6]{Aldous_FMIE} concerns the distributions of atoms of empirical measures, which correspond to the proportions of agents with different opinions. As noted in \cite{Aldous_FMIE} in the study of clusters of opinions, the entropy
process ${\sf Ent}\big({\bf m}(\xi_t)\big)$
and  diversity process ${\sf Div}\big({\bf m}(\xi_t)\big)$ associated with a voter model $(\xi_t)$ are of particular interest. Here for $\lambda\in
\ms P(S)$, $\mathfrak a(\lambda)=(\mathfrak
a_1(\lambda),\mathfrak a_2(\lambda),\dots)$ denotes the
sequence of the atom sizes of $\lambda$ arranged in
decreasing order: $\mathfrak a_1(\lambda)\geq \mathfrak
a_2(\lambda)\geq \cdots$, and the entropy  and diversity of $\lambda$ are defined by
\begin{linenomath}
\begin{align}
\text{\sf Ent}(\lambda) \doteq -& \sum_{i=1}^\infty \mathfrak a_i(\lambda)\log\big(\mathfrak a_i(\lambda)\big),\label{def:ent}\\
\label{def:div}
{\sf Div}(\lambda)\doteq &\sum_{i=1}^\infty \mathfrak a_i(\lambda)^2,
\end{align}
\end{linenomath}
respectively, with the convention that $\mathfrak a_i(\lambda)=0$ if the number of atoms of $\lambda$ is less than $i$ and $0\log 0=0$. Notice that entropy emphasizes small atoms, while diversity emphasizes large atoms.

Given \eqref{eq:empmeas}, one may expect that in the
  limit of   large $N$, the distributions of the 
entropy processes and diversity processes for voter models should be
  well approximated by the same quantities for the limiting
  Fleming-Viot process.
The main difficulty in
  obtaining a result along these lines is that, in
  general, weak convergence of a sequence of probability measures
does not imply weak
convergence of the sizes and locations of their
atoms. In particular, even though the Fleming-Viot process
without mutation is almost surely purely atomic at positive times
(cf. \cite[Theorem~4.5]{EK:MP} or \cite[Theorem~4.1]{Shiga_90}), (\ref{mN}) suggests but
does not imply convergence of the sizes and locations of atoms. On the other hand,
Ethier and Kurtz~\cite{EK:AT} found a way
  around this difficulty
two decades ago by introducing the finer mode of weak atomic
convergence. It reinforces the usual weak
convergence of measures to the effect that sizes and
locations of atoms converge in the natural way (see Section~\ref{sec:atom} or \cite[Section~2]{EK:AT} for
further discussions, and also \cite[Section~4]{EK:AT} for its genetic applications to Moran processes). In Theorem~\ref{thmm:2} below
we change the weak topology on $\ms P(S)$ to this finer topology of weak atomic convergence, and establish the weak convergence in (\ref{mN}) 
(again as processes taking values in $\ms P(S)$) with
  respect to this topology under an additional
  mild condition on the mutation measures. 

Finally, in the setting of \cite{Aldous_FMIE} in which 
mutation is absent, using the weak atomic convergence above and
duality again, we obtain convergence of the entropy processes and 
diversity processes under an additional condition on the
systems of coalescing $q^{(n)}$-Markov chains
(see
Theorem~\ref{thmm:3}). The condition states that in the
limit of large $E_n$ and with time being rescaled 
as before, the first time that the effective size of the full system of coalescing Markov chains (starting from all spatial points) reduces to a fixed number converges in distribution. The limit is given by the same limit as in the mean-field case, and can be characterized as a certain convergent series of independent exponential variables with simple parameters which has a natural interpretation that in the limit, the effective size of the coalescing system ``comes down from infinity''. This property
 holds for Moran processes through their
duality with the Kingman coalescent \cite{Kingman_1982},
and also for certain coalescing random walks on discrete tori by Cox \cite{Cox_1989}.
Its general validity under
certain mild conditions on the kernels $q^{(n)}$
was shown only recently in a
remarkable result in \cite{Oliveira_2013} (see
Section~\ref{sec:atom} for more details).  Under this
assumption on the dual 
coalescing Markov chains, the numbers of atoms at fixed positive
times under the corresponding time-rescaled voter models are bounded in probability in the limit
(Theorem~\ref{thmm:3}). The weak convergence of  the entropy processes and diversity processes is then a simple consequence of the convergence in (\ref{mN}) reinforced to the weak atomic
convergence.

In this way,
we provide an answer to the inquiry in \cite{Aldous_FMIE} in
terms of the well-established theory of Fleming-Viot
processes and via the notion of weak atomic convergence of measure-valued processes in \cite{EK:AT}.

\paragraph{\bf Organization of the paper} In
Section~\ref{sec:emvm}, we study the probability-measure-valued
functional $\mathbf m$ in (\ref{eq:empmeas}) under
generators of finite voter models. We work with two
different classes of test functions, one for
the weak convergence of the  empirical measures for voter models
and the other for their weak atomic convergence. In
Section~\ref{sec:coal}, we characterize the Fleming-Viot
processes, and proceed to some preliminary estimates for the weak convergence of  empirical measures associated with voter models
to a Fleming-Viot process after an outline of our  method. The proof of the weak convergence of empirical measures in (\ref{mN}) is then presented in
Section~\ref{sec:limit}, and the result can be found in
Theorem~\ref{thmm:1}. In Section~\ref{sec:atom}, we first
discuss briefly the Ethier-Kurtz weak atomic convergence introduced in \cite{EK:AT} and then
prove Theorem~\ref{thmm:2}. We further reinforce Theorem~\ref{thmm:2} in Theorem~\ref{thmm:3},
obtaining convergence of atom-size
point processes for voter models in the absence of mutation. Characterizations of the
limiting processes for the entropy processes and the
diversity processes then follow. Finally, we close this paper
with Section~\ref{sec:dual}. There, for the convenience of the reader, we give a brief self-contained
treatment of duality which allows us to prove the key
  probability estimate \eqref{f:bdd} in
  Section~\ref{sec:coal}.

\section{Dynamics of empirical measures for voter models}\label{sec:emvm} 
In this section, we study empirical measures under the generator of an $(E,q,\mu)$-voter model. The pair $(E,q)$ is an irreducible transition kernel with stationary distribution $\pi$
and satisfies the zero trace condition $q(x,x)\equiv 0$.
The mutation measure $\mu$ is a finite measure on $S$. These assumptions on voting kernels and mutation measures will be in force throughout the rest of this paper. 

By our description of the $(E,q,\mu)$-voter model in Section~\ref{sec:intro}, the generator of the voter model is given by
\begin{linenomath}
\begin{align}\label{eq:vmgen}
&L^{{\mu}}_{\sf VM} F(\xi) \doteq 
\sum_{x,y\in E}q(x,y)[F(\xi^{x,y})-F(\xi)]+\sum_{x\in E}\int_S[F(\xi^{x|\sigma})-F(\xi)]d\mu(\sigma),
\end{align}
\end{linenomath}
where $F: S^E\to  \R$, the configuration $\xi^{x,y}$ in the first sum of (\ref{eq:vmgen}) is given by
\begin{equation}\label{eq:xiab}
\xi^{x,y}(a) = \begin{cases}
\xi(y)&\text{if }a=x,\\
\xi(a)&\text{if }a\ne x ,
\end{cases}
\end{equation}
and the configuration $\xi^{a|\sigma}$ in the second sum of
(\ref{eq:vmgen}) is obtained by replacing the type of $\xi$
at $a$ with $\sigma$. Notice that the first sum in
(\ref{eq:vmgen}) governs the voting mechanism of the voter
model, while the second sum governs its mutation
mechanism. Since the voting mechanism is irrelevant to the
mutation mechanism, it will become convenient to handle the
two sums separately and henceforth we write
$L^\mu_{\sf VM}= L_{\sf VM}+ L_\mu $, where
\[
L_{\sf VM}=L^0_{\sf VM} 
\]
and
\begin{linenomath}
\begin{align}\label{def:Lmu}
L_\mu F(\xi)=\sum_{x\in E}\int_S[F(\xi^{x|\sigma})-F(\xi)]d\mu(\sigma).
\end{align}
\end{linenomath}

We consider two classes of test functions for the empirical measures of voter models. They will be used separately to characterize limits with respect to weak convergence or weak atomic convergence of empirical measures of voter models.
To define these test functions, first we write
\begin{linenomath}
\begin{align}\label{def:integration}
\lambda(f) = \langle f,\lambda\rangle \doteq \int_{S^k} f(\sigma_1,\cdots,\sigma_k)d\lambda(\sigma_1,\cdots,\sigma_k) ,
\end{align}
\end{linenomath}
for $\lambda\in\scrP(S^k)$, $f:S^k\to \mathbb R$ and $k\in \Bbb N$,
where the reference to $k$ for $\lambda(f)$ and $\langle f,\lambda\rangle$ will remain implicit but should be clear from the context.

The first class of test functions consists of functions $\phi$ of the form:
\begin{equation}\label{def:phik}
\phi(\lambda) = 
\prod_{i=1}^k\langle f_i,\lambda\rangle
\end{equation}
for continuous functions $f_i:S\to \Bbb R$ and $k\in \Bbb N$. We set
$\Phi_k$ to be  the set of functions on $\ms P(S)$  taking the form (\ref{def:phik})
and $\Phi=\bigcup_{k\in \Bbb N}\Phi_k$. To facilitate the following computations, we also set, for a function $\phi$ as in (\ref{def:phik}),
\begin{linenomath}
\begin{align}\label{eq:restrict}
\phi_A(\lambda) \doteq  \prod_{i\in A}\langle f_i,\lambda
\rangle\quad\mbox{and}\quad
\Delta_A^\phi(\sigma,\tau) \doteq   \prod_{i\in A}[f_i(\sigma)-
f_i(\tau)]
\end{align}
\end{linenomath}
for all $A\subseteq \{1,\cdots,k\}$, $\lambda\in \ms P(S)$ and $\sigma,\tau\in S$, with the convention that the products are identically equal to $1$ if $A$ is empty. Notice that $\Delta^\phi_A(\sigma,\tau)\neq \Delta^\phi_A(\tau,\sigma)$ in general. 

The introduction of the test functions in (\ref{def:phik}) is to facilitate the study of weak convergence of the empirical measures for voter models by the method of moments. If $\phi$ is as in (\ref{def:phik}) and $f_i$'s are indicator functions of single types, then $\phi\big(\mathbf m(\xi)\big)$ reduces to a product of moments of densities for (possibly) different types in $\xi$ (recall the notation $\mathbf m(\xi)$ in (\ref{eq:empmeas})). 
This fact should make evident our consideration of the following proposition, which characterizes dynamics of empirical measures for voter models. See also Section~\ref{sec:coal} for the characterization of Fleming-Viot processes by these test functions.

\begin{prop}\label{lem:VMgenphi} 
For every $k\in \Bbb N$ and $\phi\in \Phi_k$ taking the form (\ref{def:phik}), we have
\begin{linenomath}
\begin{align}
\begin{split}
&L_{\sf VM}\phi \circ \mathbf m(\xi) 
\\
&\hspace{1cm}= \sum_{\substack{A:A\subseteq\{1,\dots,k\}\\ |A|\ge 2} }
\phi_{A^\complement}\circ \mathbf m(\xi)  \sum_{x,y\in E}\pi(x)^{|A|}
q(x,y)\Delta^\phi_A\big(\xi(y),\xi(x)\big),\label{eq:vmgenphi1-1}\\
\end{split}\\
\begin{split}
&L_{\mu}\phi\circ \mathbf m(\xi) \\
&\hspace{1cm}=\sum_{\substack{A:A\subseteq\{1,\dots,k\}\\ |A|\ge 1} }\phi_{A^\complement} \circ\mathbf m(\xi)
\sum_{x\in E}\pi(x)^{|A|}\int_S\Delta_A^\phi\big(\sigma,\xi(x)\big)d\mu(\sigma)\label{eq:vmgenphi1-2}
\end{split}
\end{align}
\end{linenomath}
with the convention that a sum over an empty set is zero, where $\mathbf m$ is defined by (\ref{eq:empmeas}) and $\phi\circ \mathbf m(\xi)=\phi\big(\mathbf m(\xi)\big)$.
\end{prop}
\begin{proof}
We start with the proof of (\ref{eq:vmgenphi1-1}).
By the definition of $\xi^{x,y}$ in (\ref{eq:xiab}), we have, for all $f:S\to \R$,
\begin{linenomath}
\begin{align*}
\langle f,\mathbf m(\xi^{x,y})\rangle 
=\langle f,\mathbf m(\xi)\rangle +  \pi(x) [f\circ\xi(y) -
f\circ\xi(x)]
\end{align*}
\end{linenomath}
since $f\circ \xi^{x,y}$ and $f\circ \xi$ only differ at site $x$ and the difference is $f\circ \xi(y)-f\circ \xi(x)$. (Here, $f\circ \xi(x)\equiv f\big(\xi(x)\big)$.)  
It follows that
\begin{linenomath}
\begin{align}
&\phi \circ\mathbf m(\xi^{x,y})-\phi\circ \mathbf m(\xi) \notag\\
&=
\prod_{i=1}^k\Big(\langle f_i,\mathbf m(\xi) \rangle 
+ \pi(x)[f_i\circ\xi(y)-f_i\circ\xi(x)]\Big) -\phi \circ\mathbf m(\xi) 
\notag\\ & 
=
\sum_{A:A\subseteq\{1,\dots,k\}}
\prod_{i\in A^\complement}\langle f_i,\mathbf m(\xi) \rangle
\prod_{i\in A}\pi(x)[f_i\circ\xi(y)-f_i\circ\xi(x)]
  -\phi \circ\mathbf m(\xi)
\notag\\ &
= \sum_{\substack{A:A\subseteq\{1,\dots,k\}\\ |A|\ge 1} }
\phi_{A^\complement}\circ \mathbf m(\xi)\pi(x)^{|A|}\Delta_A^\phi\big(\xi(y),\xi(x)\big),\label{phim}
\end{align}
\end{linenomath}
where in the last equality we use the definition of $\phi$ and the notations in (\ref{eq:restrict}).
We deduce from the last equality and \eqref{eq:vmgen} that
\[
L_{\sf VM} \phi\circ \mathbf m(\xi) = 
\sum_{\substack{A:A\subseteq\{1,\dots,k\}\\ |A|\ge 1} }
\phi_{A^\complement}\circ \mathbf m(\xi) \sum_{x,y\in   E}\pi(x)^{|A|}q(x,y) 
\Delta^\phi_A\big(\xi(y),\xi(x)\big).
\]
In particular, all the summands on the right-hand side of (\ref{phim}) indexed by $A=\{i\}$ for $1\leq i\leq k$ vanish since
\begin{linenomath}
\begin{align*}
&\sum_{x,y\in E}\pi(x)^{|A|}q(x,y)\Delta^\phi_{A}\big(\xi(y),\xi(x)\big)\\
&\hspace{1cm}=
\sum_{x,y\in E}\pi(x)q(x,y)f_i\circ\xi(y)
- \sum_{x,y\in E}\pi(x)q(x,y)f_i\circ \xi(x)
=0,
\end{align*}
\end{linenomath}
where the last equality follows from the fact that $\pi q=\pi$ and $q\1\equiv \1$. The last two displays prove (\ref{eq:vmgenphi1-1}).

The proof of the required equality for $L_{\mu} \phi\circ \mathbf m(\xi)$ in (\ref{eq:vmgenphi1-2}) is similar. By the definition of $L_{\mu}$ in (\ref{def:Lmu}), we have
\begin{linenomath}
\begin{align*}
L_\mu \phi\circ \mathbf m(\xi)=&\sum_{x\in E}\int_S[\phi\circ\mathbf m(\xi^{x|\sigma})-\phi\circ\mathbf m(\xi)]d\mu(\sigma)\\
=&\sum_{x\in E}
\sum_{\substack{A:A\subseteq\{1,\dots,k\}\\ |A|\ge 1} }
\phi_{A^\complement} \circ\mathbf m(\xi)\pi(x)^{|A|}\int_S\Delta_A^\phi\big(\sigma,\xi(x)\big)d\mu(\sigma),
\end{align*}
\end{linenomath}
which follows from the same calculation in the display for (\ref{phim}) if one replaces $\xi(y)$ with $\sigma$ from the second line on in that display and then integrates with respect to $d\mu(\sigma)$ over $S$. The foregoing equality proves (\ref{eq:vmgenphi1-2}). 
\end{proof}

In the next section, we will begin our proof for the weak convergence of empirical measures for voter models toward Fleming-Viot processes stated in (\ref{mN}). As discussed in Section~\ref{sec:intro}, our method uses the well-posedness of the martingale problems for Fleming-Viot processes, which, roughly speaking, requires us to characterize only the first two moments of empirical measures for voter models. This direction will call for the following special cases of  Proposition~\ref{lem:VMgenphi}.

\begin{cor}\label{cor:gen12}
{\rm (1)} For $\phi$ as in (\ref{def:phik}) with $k=1$, we have
\begin{linenomath}
\begin{align}
L_{\VM}\phi\circ\mathbf m(\xi)=0\quad\mbox{and}\quad 
L_{\mu}\phi\circ\mathbf m(\xi)=\langle A_\mu f_1,\mathbf m(\xi)\rangle,
 \label{eq:vmgenphi2}
\end{align}
\end{linenomath}
where
\begin{linenomath}
\begin{align}\label{def:Amu}
A_\mu f(\tau)\doteq\int_S[f(\sigma)-f(\tau)]d\mu(\sigma).
\end{align}
\end{linenomath}

\noindent {\rm (2)} For $\phi$ as in (\ref{def:phik}) with $k=2$, we have
\begin{linenomath}
\begin{align}
L_{\sf VM} \phi \circ\mathbf m(\xi) 
=& \sum_{x,y\in E}\pi(x)^2q(x,y)
\Delta^\phi_{\{1,2\}}\big(\xi(y),\xi(x)\big),\label{eq:vmgenphi3-1}\\
\begin{split}
L_{\mu}\phi\circ\mathbf m(\xi)=&\sum_{i=1}^2 \langle A_\mu f_i,\mathbf m(\xi)\rangle\prod_{j:j\neq i}\langle f_j,\mathbf m(\xi)\rangle\\
&+\pid\langle f_1 f_2,\mu\rangle-\langle f_1,\mu\rangle \sum_{x\in E}f_2\circ \xi(x)\pi(x)^2\\
&-\sum_{x\in E}f_1\circ \xi(x)\pi(x)^2\langle f_2,\mu\rangle +\mu(\1)\sum_{x\in E}f_1f_2\circ \xi(x)\pi(x)^2,\label{eq:vmgenphi3-2}
\end{split}
\end{align}
\end{linenomath}
where $\pid=\sum_{x\in E}\pi(x)^2$ and $f\circ \xi(x)=f\big(\xi(x)\big)$.
\end{cor}

Next we consider the following test functions for empirical measures of voter models:
\begin{linenomath}
\begin{align}\label{def:Ff}
F_f(\xi)\doteq \langle f,\mathbf m(\xi)^{\otimes 2}\rangle
=\sum_{x,y\in E}\pi(x)\pi(y)f\big(\xi(x),\xi(y)\big)
\end{align}
\end{linenomath}
for Borel measurable functions $f:S\times S\to \R$.
These test functions $F_f$ will
be used in Section~\ref{sec:atom} to investigate weak atomic convergence for empirical measures of voter models, where a criterion from \cite{EK:AT} is invoked (see the proof of Theorem~\ref{thmm:2}).

\begin{prop}\label{prop:genF}
For $F_f$ defined as in (\ref{def:Ff}) with respect to a Borel measurable function $f:S\times S\to \R$, we have
\begin{linenomath}
\begin{align}
\begin{split}
&L_{\sf VM}F_f(\xi)
=\sum_{x,y\in E}\pi(x)^2q(x,y)\big[f\big(\xi(y),\xi(y)\big)-f\big(\xi(x),\xi(x)\big)\big]\\
&-\sum_{x,y\in E}\pi(x)^2q(x,y)\left[f\big(\xi(y),\xi(x)\big)+f\big(\xi(x),\xi(y)\big)-2f\big(\xi(x),\xi(x)\big)\right]
\label{eq:Ff1}
\end{split}
\end{align}
\end{linenomath}
and
\begin{linenomath}
\begin{align}
\begin{split}\label{eq:Ff2}
L_\mu F_f(\xi)=&\sum_{x\in E}\pi(x)^2\int_S\big[f\big(\sigma,\sigma\big)-f\big(\xi(x),\xi(x)\big)\big]d\mu(\sigma)\\
&\hspace{-1.5cm}-\sum_{x\in E}\pi(x)^2\int_S\left[
f\big(\xi(x),\sigma\big)+
f\big(\sigma,\xi(x)\big)-2f\big(\xi(x),\xi(x)\big)\right]d\mu(\sigma)\\
&\hspace{-1.5cm}+\sum_{x,y\in E}\pi(x)\pi(y)\int_S\left[f\big(\xi(x),\sigma\big)+f\big(\sigma,\xi(y)\big)-2f\big(\xi(x),\xi(y)\big)\right]d\mu(\sigma).
\end{split}
\end{align}
\end{linenomath}
\end{prop}
\begin{proof}
We compute
\begin{linenomath}
\begin{align}
L_{\VM}F_f(\xi)
=&\sum_{x,y\in E}q(x,y)\sum_{a,b\in E}\pi(a)\pi(b)\left[f\big(\xi^{x,y}(a),\xi^{x,y}(b)\big)-f\big(\xi(a),\xi(b)\big)\right]\notag\\
=&\sum_{x,y\in E}\sum_{a,b:a= x\mbox{\;\tiny or }b= x}q(x,y)\pi(a)\pi(b)\left[f\big(\xi^{x,y}(a),\xi^{x,y}(b)\big)-f\big(\xi(a),\xi(b)\big)\right]\notag\\
\begin{split}\notag
=&\sum_{x,y\in E}q(x,y)\pi(x)^2\left[f\big(\xi(y),\xi(y)\big)-f\big(\xi(x),\xi(x)\big)\right]\\
&+\sum_{x,y\in E}\sum_{b:b\neq x}q(x,y)\pi(x)\pi(b)\left[f\big(\xi(y),\xi(b)\big)-f\big(\xi(x),\xi(b)\big)\right]\\
&+\sum_{x,y\in E}\sum_{a:a\neq x}q(x,y)\pi(a)\pi(x)\left[f\big(\xi(a),\xi(y)\big)-f\big(\xi(a),\xi(x)\big)\right].
\end{split}
\end{align}
\end{linenomath}
The required equation (\ref{eq:Ff1}) now follows from the preceding equality, if we apply again the fact that $\pi q=\pi$ and $q\mathbf \1=\1$ to its last two terms. The proof of (\ref{eq:Ff2}) is similar, and is left to the readers. 
\end{proof}

\section{Fleming-Viot processes and empirical measures of voter models}\label{sec:coal} 
We characterize Fleming-Viot processes by the following martingale-problem formulation. Recall the notations  in (\ref{def:integration}), (\ref{eq:restrict}) and (\ref{def:Amu}), and the classes of test functions $\Phi_k$ and $\Phi$ introduced in Section~\ref{sec:emvm}.
We define an operator $L^\mu_{\sf FV}$ on $\Phi$ by
\begin{linenomath}
\begin{align}\label{eq:FVgen2}
\begin{split}
L_{\sf FV}^{\mu} \phi(\lambda)\doteq \frac{1}{2}\sum_{1\le i<j\le k}\big\langle \Delta^{\phi}_{\{i,j\}},\lambda^{\otimes 2}\big\rangle
\prod_{\ell: \ell\ne i,j} \langle f_\ell,\lambda\rangle+\sum_{1\leq i\leq k}\langle A_\mu f_i,\lambda\rangle\prod_{\ell:\ell\neq i}\langle f_\ell,\lambda\rangle  
\end{split}
\end{align}
\end{linenomath}
for $\lambda\in\scrP(S)
$ and $\phi\in \Phi_k$ as in \eqref{def:phik}, and put $L_{\sf FV}=L^0_{\sf FV}$ to be consistent with the analogous notation for generators of voter models. Then the $\mu$-Fleming-Viot process $X$ can be characterized by the well-posed martingale problem: for all $\phi\in \Phi$, the process
\begin{linenomath}
\begin{align}\label{def:FV}
\phi(X_t)-\phi(X_0)-\int_0^t L^\mu_{\sf FV}\phi(X_s)ds
\end{align}
\end{linenomath}
is a continuous martingale (cf. \cite[Section 10.4]{EK:MP}).

Our program in the rest of this section 
 is to obtain a quantitative version of  the approximation:
\begin{linenomath}
\begin{align}\label{VMFV-approx}
\E_\xi[L_{\VM}\phi\circ \mathbf
m(\xi_s)]\simeq \mbox{constant}\cdot
L_{\FV}\phi\big( \mathbf m(\xi)\big) 
\end{align}
\end{linenomath} 
for enough large $s$ when suitable voting kernels are used.
We will give a rigorous form of this approximation for $\phi\in \Phi_2$ that is adequate to obtain the weak convergence of empirical measures for voter models in Section~\ref{sec:limit}, as well as a bound for the left-hand side of (\ref{VMFV-approx}) to gain its control at ineligible times $s$. 

Before stating and proving these results we need a preliminary inequality (Proposition~\ref{prop:dual} below). It requires the standard duality which connects the voter model and a system of coalescing $q$-Markov chains $\{X^x;x \in E\}$. Here, $X^x$ are rate-$1$ $q$-Markov chains starting at $x\in E$, and they move independently before meeting and together afterward. Let $M_{x,y}=\inf\{t\geq 0;X^x_t=X^y_t\}$ be this meeting time. In Section \ref{sec:dual}, we present in detail the graphical construction of this duality relation. In particular, for any fixed $t>0$ and $x,y \in E$, the construction provides an explicit construction of $\xi_t(x)$ and $\xi_t(y)$ in terms of coalescing chains $X^x$ and $X^y$ started at $x$ and $y$, respectively, run for time $t$, with mutation events occurring at rate $\mu(\1)$ along the paths of these chains. Here in Proposition~\ref{prop:dual} below and what follows, we write
\begin{linenomath} \begin{align}\label{def:mubar}
\overline{\mu}(\Sigma)=\frac{\mu(\Sigma)}{\mu(\1)}
\end{align}
\end{linenomath}
with the convention that $\tfrac{0}{0}=0$ when $\mu=0$, and $\E_\xi$ and $\P_\xi$ for expectation and probability of the voter model with generator $L^\mu_{\sf VM}$ (\ref{eq:vmgen}) and initial state $\xi\in S^E$, respectively. In addition, with respect to a system of coalescing $q$-chains $\{X^x;x\in E\}$ as above, we write $M_{x,y}=\inf\{t\geq 0;M^x_t=M^y_t\}$.

\begin{prop}\label{prop:dual} 
Given $f:S\times S\to \R$  satisfying
$f(\sigma,\sigma)\equiv 0$ and
$\sup_{\sigma,\tau}|f(\sigma,\tau)|\leq 1$, we have, for  all $\xi\in \{0,1\}^E$, $x,y\in E$ and $t\in \R_+$,
\begin{linenomath}
\begin{align}\label{f:bdd}
\begin{split}
&\Big|\E_\xi\left[f\big(\xi_t(x),\xi_t(y)\big)\right]-\E\left[f\big(\xi(X^x_t),\xi(X^y_t)\big)\right]\Big|\\
&\hspace{2cm}\leq 
C\big(1-e^{-\mu(\1)t}\big)\P(M_{x,y}>t)+C\mu(\1)\int_0^t \P(M_{x,y}>s)ds
\end{split}
\end{align}
\end{linenomath}
for some universal constant $C$. 
\end{prop}

The proof of Proposition~\ref{prop:dual} will be given in Section~\ref{sec:dual}, where we use the graphical construction mentioned above. Roughly speaking, it comes from the following observation. In order for the difference 
\[
f\big(\xi_t(x),\xi_t(y)\big)-f\big(\xi_0(X^x_t),\xi_0(X^y_t)\big)
\]
to be nonzero one of the following events must occur: (a) $M_{x,y} > t$ and no mutation events occur along the two paths before $t$, or (b) $M_{x,y}\leq t$ and at least one mutation event occurs along the paths before $M_{x,y}$. The probability of (a) is bounded by the first term on the right-hand side of (\ref{f:bdd}), and the probability
of (b) is bounded by the expected number of mutation events before $M_{x,y}\wedge t$, which is bounded by the second term on the right-side of (\ref{f:bdd}).

Now we proceed to a rigorous form of (\ref{VMFV-approx}). To compare the dynamics of empirical measures for voter models with the dynamics of Fleming-Viot processes, we first introduce some notation. 
For an irreducible transition kernel $(E,q)$, let $(q_t)$ be the semigroup of the rate-$1$ $q$-Markov chain on $E$. Then we set 
\begin{linenomath}
\begin{align}\label{TV}
d_E(t)\doteq \max_{x\in E}\|q_t(x,\cdot)-\pi\|_{\rm TV}
\end{align}
\end{linenomath}
for the maximal total variation distance of the semigroup $(q_t)$ from its stationary distribution $\pi$, and consider the choice of mixing time:
\begin{linenomath}
\begin{align}
{\bf t}_{\rm mix}\doteq \inf\left\{t\geq 0;d_E(t)\leq (2e)^{-1}\right\}.
\end{align}
\end{linenomath}
Here in (\ref{TV}) $\|\cdot\|_{\rm TV}$ refers to the usual total variation distance.
In addition, if $q$ is reversible, we write
$\mathbf g$ for the spectral gap of 
the voting kernel $q$, which is the difference between the
largest and the second largest eigenvalues of the
symmetric matrix $\pi(x)q(x,y)$ indexed by points of $E$.
To minimize the use of sum notation, we write $(V,V')$ for a random vector taking values in $E\times E$, which is independent of $\{X^x\}$ and the voter model, and has joint distribution given by
\[
\bfP(V=x,V'=y)\equiv\frac{\pi(x)^2q(x,y)}{\pid}
\]
(recall that $\pid=\sum_{x\in E}\pi(x)^2$).
Finally, $C_\phi$ stands for a strictly positive constant, which may vary from line to line and depends only on a function $\phi\in \Phi$.

\begin{prop}\label{prop:main}
For $\phi\in\Phi_2$ and $s,t\in (0,\infty)$ with $s<t$, we have the following two different estimates between the generators $L_{\sf VM}$ and $L_{\sf FV}$ where mutation is absent:
\begin{linenomath}
\begin{align}
&\sup_{\xi\in S^E}\Big| \E_\xi \left[
L_{\sf VM}\hphi(\xi_t)\right] - 2\pi_{\rm diag}
\bfP(M_{V,V'}>s) L_{\bf \FV} \phi\big(\mathbf m(\xi)\big)\Big|\notag
\\
\begin{split}\label{eq:main0} 
&\le C_\phi
\pid\P\big(M_{V,V}\in (s,t]\big)+C_\phi
\pid\big(1-e^{-\mu(\1) t}\big)\P\big(M_{V,V}>t\big)\\
&\hspace{.5cm}+C_\phi
\pid\mu(\1)\int_0^t \P(M_{V,V'}>r)dr+C_\phi d_E(t-s)\pid \P(M_{V,V'}>t),
\end{split}
\end{align}
\end{linenomath}
and
\begin{linenomath}
\begin{align}
&\sup_{\xi\in S^E}\Big| \E_\xi \left[
L_{\sf VM}\hphi(\xi_t)\right] - 2\pi_{\rm diag}
\bfP(M_{V,V'}>s) L_{\bf \FV} \phi\big( \mathbf m(\xi)\big)\Big|\notag\\
\begin{split}
&\le C_\phi
\pid\P\big(M_{V,V}\in (s,t]\big)+C_\phi\pid \big(1-e^{-\mu(\1) t}\big)\P(M_{V,V'}>t)\\
&\hspace{.5cm}+C_\phi\pid\mu(\1)\int_0^t \P(M_{V,V'}>r)dr
+C_\phi \pi_{\max}
e^{-\mathbf g(t-s)},\label{eq:main1}
\end{split}
\end{align}
\end{linenomath} 
where $\pi_{\max}=\max_{x\in E}\pi(x)$. We stress that the expectation $\E$ and probability $\P$ in the foregoing two displays are for the voter model with generator $L^\mu_{\sf VM}$ (\ref{eq:vmgen}).
\end{prop}

The inequalities in Proposition~\ref{prop:main} will be used to formalize the approximation (\ref{VMFV-approx}) by making
\begin{align}\label{Mapp-VV}
2\pi_{\rm diag}
\bfP(M_{V,V'}>s)\to 1 
\end{align}
with appropriate choices of $s$ and voting kernels. See the proof of Step~\ref{s:2} of Lemma~\ref{lem:qvbnd} for further details.

\begin{proof}[Proof of Proposition~\ref{prop:main}]
As in the proof of \cite[Proposition~6.1]{CCC}, 
first we derive a preliminary estimate for $\E_\xi[L_{\sf VM}\phi\circ \mathbf m(\xi)]$ which will be refined to obtain (\ref{eq:main0}) and (\ref{eq:main1}).

Now we use the explicit form (\ref{eq:vmgenphi3-1}) of $L_{\sf VM}\phi\circ \mathbf m(\xi)$ for $\phi\in \Phi_2$ taking the form (\ref{def:phik}),
Proposition~\ref{prop:dual} and the fact that $\Delta^\phi_{\{1,2\}}\big(\xi(x),\xi(x)\big)\equiv 0$. They give
\begin{linenomath}
\begin{align}
&\E_\xi\left[L_{\sf VM}\phi\circ \mathbf m(\xi_t)\right]=\pid\E\left[\Delta^\phi_{\{1,2\}}\big(\xi_t(V),\xi_t(V')\big)\right]\notag\\
\begin{split}\label{eq:main2}
=&\pid\E\left[\Delta^\phi_{\{1,2\}}\big(\xi(X^V_t),\xi(X^{V'}_t)\big);M_{V,V'}>s\right]
+\pid\vep_0(t,\xi),
\end{split}
\end{align}
\end{linenomath}
where the term $\vep_0(t,\xi)$ satisfies the inequality
\begin{linenomath}
\begin{align}\label{bdd:vep0}
\begin{split}
|\vep_0(t,\xi)|\leq &C_\phi \Big(\P(M_{V,V'}\in (s,t])+\big(1-e^{-\mu(\1) t}\big) \P(M_{V,V'}>t)\\
&\hspace{4cm}+\mu(\1)\int_0^t\P(M_{V,V'}>r)dr\Big).
\end{split}
\end{align}
\end{linenomath}
In the rest of the proof, we will consider two different estimates for the first term on the right-hand side of (\ref{eq:main2}).

For the first estimate, we
apply the Markov property of $(X^V,X^{V'})$ at time $s$ to get
\begin{linenomath}
\begin{align}
&\pi_{\rm diag}\E \left[\Delta^\phi_{\{1,2\}}\big(\xi(X^V_t),\xi(X^{V'}_t)\big); M_{V,V'}>s\right] \notag\\
=& \pi_{\rm diag}\sum_{x,y\in E}\bfP(X^V_s=x,X^{V'}_s=y,M_{V,V'}>s) \E\left[\Delta^\phi_{\{1,2\}}\big(\xi(X^x_{t-s}),\xi(X^y_{t-s})\big)\right]\notag\\
=&\pi_{\rm diag} \sum_{x,y\in E}\bfP(X^V_s=x,X^{V'}_s=y,M_{V,V'}>s)\notag\\
&\times\Bigg\{\sum_{a,b\in
  E}[q_{t-s}(x,a)q_{t-s}(y,b)-\pi(a)\pi(b)]\Delta^\phi_{\{1,2\}}\big(\xi(a),\xi(b)\big) + 2 L_\FV\phi\big(\mathbf m(\xi)\big)\Bigg\}\notag\\
=&
\pid\vep_1(s,t,\xi)+ 2\pi_{\rm diag} \bfP(M_{V,V'}>s)L_\FV\phi\big(\mathbf m(\xi)\big),\label{main-0}
\end{align}
\end{linenomath}
where the second equality follows from the definition
\eqref{eq:FVgen2} of $L_{\FV}=L_{\sf FV}^0$ and the
definition of $\Delta^\phi_{\{1,2\}}$ in
(\ref{eq:restrict}), and in the last equality the error term
$\vep_1(s,t,\xi)$ satisfies the inequality
\begin{linenomath}
\begin{align*}\begin{split}
&|\vep_1(s,t,\xi)|\le C_\phi d_E(t-s) \bfP(M_{V,V'}>s).
\end{split}
\end{align*}
\end{linenomath}
(cf. \cite[Proposition~4.5]{Levin_2008}). 
Applying (\ref{main-0}) to (\ref{eq:main2}), we
  obtain (\ref{eq:main0}).

To obtain the second estimate (\ref{eq:main1}), now we consider the difference
\begin{linenomath}
\begin{align}
&\pi_{\rm diag} \E\left[\Delta^\phi_{\{1,2\}}\big(\xi(X^V_t),\xi(X^{V'}_t)\big);M_{V,V'}>s\right]-2\pid\P(M_{V,V'}>s) L_{\sf FV}\phi\big(\mathbf m(\xi)\big)\notag\\
\begin{split}
=&\pid\E\Bigg[\Delta^\phi_{\{1,2\}}\big(\xi(X^V_t),\xi(X^{V'}_t)\big)\\
&\hspace{2cm}-\sum_{x,y\in E}\pi(x)\pi(y)\Delta^\phi_{\{1,2\}}\big(\xi(x),\xi(y)\big);M_{V,V'}>s\Bigg].\label{vv1}
\end{split}
\end{align}
\end{linenomath}
Taking into account the expansion
\[
\Delta^\phi_{\{1,2\}}(\sigma,\tau)=f_1(\sigma)f_1(\sigma)-f_1(\sigma)f_2(\tau)-f_1(\tau)f_2(\sigma)+f_2(\sigma)f_2(\tau)
\]
for $\phi(\lambda)\equiv \langle f_1,\lambda\rangle\langle f_2,\lambda\rangle$,
we see that (\ref{vv1}) and Markov property
imply
\begin{linenomath}
\begin{align*}
&\left|\pi_{\rm diag} \E\left[\Delta^\phi_{\{1,2\}}\big(\xi(X^V_t),\xi(X^{V'}_t\big);M_{V,V'}>s\right]-2\pid\P(M_{V,V'}>s) L_{\sf FV}\phi\big( \mathbf m(\xi)\big)\right|\\
\leq &C_\phi\pid \sum_{j,k\in \{1,2\},W\in \{V,V'\}}\E\left[\big|q_{t-s}(f^k_j\circ\xi)(X^{W}_s)-\pi(f_j^k\circ\xi)\big|\right]\\
\leq &C_\phi \pid\sum_{j,k\in \{1,2\},W\in \{V,V'\}}\E\left[\big(q_{t-s}f_j^k\circ\xi(X^{W}_s)-\pi(f_j^k\circ\xi)\big)^2\right]^{1/2}\\
\leq & C_\phi\pid\sum_{j,k=1}^2\left(\frac{\pi_{\max}}{\pid}{\rm Var}_\pi(f_j^k)\right)^{1/2}e^{-\mathbf g(t-s)},
\end{align*}
\end{linenomath}
where $f_j^k$ is the $k$-th power of $f_j$, ${\rm Var}_\pi(f) = \sum_{x\in
    E}\big(f(x)-\pi(f)\big)^2\pi(x)$, and the last inequality follows from
the definition of $(V,V')$ and a standard variance bound for the mixing times of Markov chains
\cite[Lemma~2.4]{Diaconis/Saloff Coste}.
Applying the foregoing inequality to (\ref{eq:main2}) gives (\ref{eq:main1}). The proof is complete.
\end{proof}

We close this section with a bound for  empirical measures of voter models, which will be used to control the left-hand side of (\ref{VMFV-approx}) at times $s$ when we cannot validate (\ref{Mapp-VV}). See the proof of Lemma~\ref{lem:qvbnd} for further details.

\begin{prop}\label{prop:genbdd}
For $\phi\in \Phi$ and $t\in (0,\infty)$, it holds that 
\begin{linenomath}
\begin{align}
\begin{split}
\sup_{\xi\in S^E}\E_\xi \left[\big|L_{\sf VM} \hphi(\xi_{t})\big|
\right]
&\leq C_\phi\pid \P(M_{V,V'}>t)\\
&\hspace{1cm}+C_\phi\pid \mu(\1)\int_0^t \P(M_{V,V'}>r)dr.\label{eq:gen}
\end{split}
\end{align}
\end{linenomath}
\end{prop}
\begin{proof}
To obtain (\ref{eq:gen}), we first claim that for $\phi\in \Phi$ as in (\ref{def:phik}),
\begin{linenomath}
\begin{align}\label{eq:vmgenbound}
\begin{split}
|L_{\sf VM} \phi \circ\mathbf m(\xi) | \leq 
C_\phi  \sum_{x,y \in E}\pi(x)^2q(x,y) \1_{\{\xi(x)\ne\xi(y)\}}.
\end{split}
\end{align}
\end{linenomath}
To see this, notice that for all $A\subseteq \{1,\dots,k\}$ with $|A|\ge 2$,
\begin{linenomath}
\begin{align}\label{eq:genbdd0}
\begin{split}
&\sum_{x,y\in E}\pi(x)^{|A|}q(x,y)|\Delta^\phi_A\big(\xi(x),\xi(y)\big)|\le C_\phi \sum_{x,y\in
  E}\pi(x)^2q(x,y)\1_{\{\xi(x)\ne \xi(y)\}},
  \end{split}
\end{align}
\end{linenomath}
and so (\ref{eq:vmgenbound}) follows from (\ref{eq:vmgenphi1-1}). 

Second we can repeat the argument for 
(\ref{eq:main2}) and obtain from Proposition~\ref{prop:dual} that, for all $x,y\in E$,
\begin{linenomath}
\begin{align*}
\sup_{\xi\in S^E}\P_\xi\big(\xi_t(x)\neq \xi_t(y)\big)\leq (C+1)\P(M_{x,y}>t)+C\mu(\1)\int_0^t \P(M_{x,y}>r)dr.
\end{align*}
\end{linenomath}
Now (\ref{eq:gen}) follows upon applying the foregoing inequality to (\ref{eq:vmgenbound}).
\end{proof}

\section{Weak convergence of empirical measures for voter models}\label{sec:limit}
The goal of this section is to prove Theorem~\ref{thmm:1}, which concerns weak convergence of empirical measures for voter models. Throughout this section and Section~\ref{sec:atom}, we consider a sequence of $(E_n,q^{(n)},\mu_n)$-voter models with
$N_n=\#E_n\nearrow \infty$.
As before, $(E_n,q^{(n)})$ are irreducible and $\mu_n$ are finite measures on the same type space $S$.

We continue to use the notations introduced in Section~\ref{sec:emvm} and \ref{sec:coal}, except that they now carry either subscripts `$n$' or superscripts `$(n)$'. We also need here the parameter $\gamma_n$ to time-change the $(E_n,q^{(n)},\mu_n)$-voter model, which is chosen to be the expected first meeting times of two independent $q^{(n)}$-Markov chains starting from stationarity (see also \cite[Theorem~2.2]{CCC}). The measure-valued processes
\begin{linenomath}
\begin{align}\label{timechange}
X^{(n)}_t=\mathbf m(\xi_{\gamma_nt})
\end{align}
\end{linenomath}
for the $(E_n,q^{(n)},\mu_n)$-voter models are the central object of this section.

\begin{thm}\label{thmm:1}
Let $\ms P(S)$ be equipped with the Prohorov metric.
Suppose that 
\begin{enumerate}
\item [\rm (i)] $\displaystyle\lim_{n\to\infty}\pid^{(n)}=0 $,
\item [\rm (ii)] $\tmeetn\mu_n$ converges weakly to a finite measure $\mu$ on $S$,
\item [\rm (iii)] $\displaystyle X_0^{(n)}\convdn \widetilde{X}_0$ in $\ms P(S)$,
\end{enumerate}
and one of the following conditions holds: 
\begin{enumerate}
\item [\rm (iv-1)] $
  \displaystyle \lim_{n\to\infty} 
\frac{\tmixn}{\tmeetn}=0$, 
\item [\rm (iv-2)] the $q^{(n)}$-chains are
  reversible with $\displaystyle \lim_{n\to\infty}
\dfrac{\log\big(e\vee\tmeetn\pinmax \big)}{\mathbf g_n\tmeetn}=0$.
\end{enumerate}
If 
$X$ is the $\mu$-Fleming-Viot process such that its initial condition has the same distribution as $\widetilde{X}_0$, then we have
\begin{equation}\label{eq:FVvmlimit}
X^{(n)}\To X\quad 
\text{ in } D\big(\R_+,\scrP(S)\big)
\end{equation}
with respect to Skorokhod's $J_1$-topology. 
\end{thm}

Let us set up some notation to facilitate the proof of Theorem~\ref{thmm:1}. We write $\F_t^{(n)}=\sigma\big(\xi_{\gamma_ns};s\leq t\big)$ for $t\in \R_+$, so that $\big(\F_t^{(n)}\big)$ is the natural filtration of the time-changed process $(\xi_{\gamma_n t})$. We decompose $\phi(X^{(n)})$, 
for any $\phi\in \Phi$, into 
\begin{linenomath}
\begin{align}\label{dec}
\phi(X^{(n)}_t)=\phi\big(X^{(n)}_0\big)+\phi\big(A^{(n)}_t\big)+\phi\big(M^{(n)}_t\big),
\end{align}
\end{linenomath}
where $A^{(n)}$ is defined by
\begin{linenomath}
\begin{align}\label{def:An}
\phi\big(A^{(n)}_t\big)\doteq \gamma_n\int_0^t L^{\mu_n}_\VM\phi\circ \mathbf m(\xi_{\gamma_n r})dr
\end{align}
\end{linenomath}
and $M^{(n)}$ is by (\ref{dec}).
Note that $\phi(M^{(n)})$ is a martingale by the definition of $\phi(A^{(n)})$ and the boundedness of $\phi$.

The following lemma is the first step of the proof of
Theorem~\ref{thmm:1}. It identifies the functional form of the quadratic variation for the empirical measure of a voter model in the limit of large $E_n$.

\begin{lem}\label{lem:qvbnd}
Under the assumptions of Theorem~\ref{thmm:1}, we have, for
any $\phi_1(\lambda)=\langle
f_1,\lambda\rangle,\phi_2(\lambda)=\langle
f_2,\lambda\rangle\in \Phi_1$ for $f_1,f_2\in \C(S)$ and
$t\in (0,\infty)$,
\begin{linenomath}
\begin{align}
&\lim_{n\to\infty}\En\left[\,\left|\int_0^t \gamma_nL^{\mu_n}_{\VM}\phi_1\circ \mathbf m(\xi_{\gamma_nr})dr-\int_0^t L_{\sf FV}^{\mu}\phi_1(X^{(n)}_r)dr
\right|\,\right]= 0,\label{eq:conv1-int}\\
&\lim_{n\to\infty}\En\left[\,\left|
\int_0^t\gamma_n L_{\sf VM}^{\mu_n}\phi_1\phi_2\circ \mathbf m(\xi_{\gamma_nr})dr-
\int_0^t L_{\sf FV}^{\mu}\phi_1\phi_2(X^{(n)}_{r})dr
\right|\,\right]= 0,\label{eq:conv2-int}
\end{align}
where $X^{(n)}$ under $\E^{(n)}$ has initial condition $X^{(n)}_0$ satisfying condition (iii) of Theorem~\ref{thmm:1}.
\end{linenomath}
\end{lem}
\begin{proof}
We fix $t\in (0,\infty)$ throughout the following and divide the proof into Step~\ref{s:1}--\ref{s:3}. We prove
(\ref{eq:conv1-int}) in Step~\ref{s:1} and (\ref{eq:conv2-int}) in Step~\ref{s:3}.

\begin{step}\label{s:1} To see (\ref{eq:conv1-int}), we note
  that by Corollary~\ref{cor:gen12} (1) and
  (\ref{eq:FVgen2}),
\begin{linenomath}
\begin{align*}
\int_0^t \gamma_nL^{\mu_n}_{\VM}\phi_1\circ \mathbf m(\xi_{\gamma_nr})dr-\int_0^t L_{\sf FV}^{\mu}\phi_1(X^{(n)}_r)dr=\int_0^t \big\langle \left(\gamma_n A_{\mu_n}-A_\mu\right)f_1,X^{(n)}_r\big\rangle dr.
\end{align*}
\end{linenomath}
The required convergence follows upon applying condition (ii) of Theorem~\ref{thmm:1} to the right-hand side of
the foregoing equality.   
\end{step}

\begin{step}\label{s:2}
We will show that it is possible to choose a sequence $(s'_n)$ such that $s_n'\nearrow \infty$, $s'_n=o(\gamma_n)$
and
\begin{linenomath}
\begin{equation}\label{eq:keylem}
\vep_n \equiv 
\sup_{\xi\in S^{E_n}}\Big| \En_\xi\left[
\gamma_nL_{\sf VM} \phi_1\phi_2\circ \mathbf m\big(\xin_{\gamma_n\cdot 2\delta_n}\big)\right] -
L_\FV\phi_1\phi_2\big(\mathbf m(\xi)\big)
\Big|\xrightarrow[n\to\infty]{} 0
\end{equation}
\end{linenomath}
for
$\delta_n=s'_n/\gamma_n$. As in the proof of the
  analogous fact given in \cite[Lemma~6.2]{CCC}, a key role
  is played by the following limiting properties of the
  meeting time laws $\P^{(n)}(M_{V,V'}\in
  \,\cdot\,)$. If the conditions (i) and either of (iv-1) or
  (iv-2) in Theorem~\ref{thmm:1} hold, then 
\begin{linenomath}
\begin{align}\label{eq:exptail}
\lim_{n\to\infty}2\gamma_n\pid^{(n)}\int_0^t \P^{(n)}(M_{V,V'}>\gamma_nr)dr=1-e^{-t}\quad \forall\;t\geq 0
\end{align}
\end{linenomath}
(see \cite[Theorem~4.1]{CCC} and the proof of
\cite[(6.13)]{CCC}). Furthermore, the argument proving
\cite[6.14]{CCC}, which assumes only conditions (i) and
(iv-1) of Theorem~\ref{thmm:1}, shows how to obtain a sequence $(s_n')$ such that
$\delta_n=s_n'/\gamma_n\to 0$, $s'_n/\tmixn\to\infty$ so that
\[
d_{E_n}(s'_n)\le \exp(-\lfloor s'_n/\tmixn\rfloor)\xrightarrow[n\to\infty]{} 0 
\]
(see \cite[Section~4.5]{Levin_2008} for
  this inequality) and also
\begin{linenomath}\begin{equation}\label{eq:snprimelimits}
2\gamma_n\pid^{(n)}\P^{(n)}(M_{V,V'}>s_n')\xrightarrow[n\to\infty]{} 1\quad \&\quad  
2\gamma_n\pid^{(n)}\P^{(n)}\big(M_{V,V'}\in(s_n',2s_n']\big)\xrightarrow[n\to\infty]{} 0.
\end{equation}
\end{linenomath}

To see our claim (\ref{eq:keylem}), we  set $s=s_n'$ and $t=2s_n'$, and use the bound in
(\ref{eq:main0}). After rearrangements, we have
\begin{linenomath}
\begin{align*}
&\hspace{.2cm}\vep_n \leq  
C_{\phi_1\phi_2}\big|2\gamma_n\pid^{(n)}\P^{(n)}(M_{V,V'}>s_n')-1\big|\\
&+
C_{\phi_1\phi_2}\gamma_n 
\pid^{(n)}\P^{(n)}\big(M_{V,V}\in (s_n',2s_n']\big)\\
&+
C_{\phi_1\phi_2} \gamma_n \pid^{(n)}
  \P^{(n)}(M_{V,V'}>s_n')d_{E_n}(s_n')
\\
&+C_{\phi_1\phi_2}
  \frac{1-e^{-\mu_n(\1) \cdot 2s_n'}}{\mu_n(\1)
  s_n'}
\gamma_n\pid^{(n)} \P^{(n)}\big(M_{V,V}>2s_n'\big)
\cdot\delta_n \gamma_n \mu_n(\1)
\\ 
&+C_{\phi_1\phi_2}\gamma_n\mu_n(\1)\cdot 
\gamma_n\pid^{(n)}\int_0^{2\delta_n} \P^{(n)}(M_{V,V'}>\gamma_n r)dr.
\end{align*}
\end{linenomath}
Here, the first three terms tend to 0 by \eqref{eq:snprimelimits}
and $d_{E_n}(s'_n)\to 0$. The fourth term tends to 0 by
\eqref{eq:snprimelimits}, condition (ii) of Theorem~\ref{thmm:1}, and
$\delta_n\to 0$.  The fifth term tends to 0 by condition (ii) of Theorem~\ref{thmm:1}, \eqref{eq:exptail}, and $\delta_n\to 0$.
Replacing (iv-1) of Theorem~\ref{thmm:1} with (iv-2) in the same theorem and using
  \eqref{eq:main1} instead of \eqref{eq:main0}, the proof of
  (\ref{eq:keylem}) follows from a similar
argument (see the second part of the proof of \cite[Lemma~6.2]{CCC}, especially (6.17) there).
\end{step}

\begin{step}\label{s:3}
We prove
  (\ref{eq:conv2-int}) in this step. Recall the definitions
  (\ref{eq:vmgen}) of $L^{\mu_n}_{\sf VM}=L_{\sf VM}+ L_{\mu_n}$ and 
  (\ref{eq:FVgen2}) of $L^\mu_{\sf FV}=L_{\sf
      FV}+L_{\mu}$. We handle the mutation
  contributions to 
   (\ref{eq:conv2-int}) first; more precisely,
  we will show that
\begin{linenomath} 
\begin{equation}\label{mut-gen}
\lim_{n\to\infty}\En\Bigg[\,\Bigg|
\int_0^t\gamma_n L_{\mu_n}\phi_1\phi_2\circ \mathbf
m(\xi_{\gamma_nr})dr
-\int_0^t L_{\mu_n}\phi_1\phi_2(X^{(n)}_r)dr\Bigg|\,\Bigg]= 0.
\end{equation}
\end{linenomath}
The foregoing equality can be obtained by an approximation
of the difference of the integrands pointwise in
$\xi$. Indeed by Corollary~\ref{cor:gen12} (2), for any $\xi\in S^{E_n}$ this difference is
\begin{linenomath}
\begin{align*}
&\gamma_n L_{\mu_n}\phi_1\phi_2\circ \mathbf m(\xi) - \sum_{i=1}^2\langle A_{\mu_n}f_i,\mathbf m(\xi)\rangle\prod_{\ell:\ell\neq i}\langle f_\ell,\mathbf m(\xi)\rangle\\
=&\sum_{i=1}^2\langle (\gamma_n A_{\mu_n}-A_\mu)f_i,\mathbf m(\xi)\rangle\prod_{\ell:\ell\neq i}\langle f_\ell,\mathbf m(\xi)\rangle\\
&+\pid^{(n)}\langle f_1 f_2,\gamma_n\mu_n\rangle-\langle f_1,\gamma_n\mu_n\rangle \sum_{x\in E_n}f_2\circ \xi(x)\pi^{(n)}(x)^2\\
&-\sum_{x\in E_n}f_1\circ\xi(x)\pi^{(n)}(x)^2\langle f_2,\gamma_n\mu_n\rangle 
+\gamma_n\mu_n(\1)\sum_{x\in E_n}f_1f_2\circ \xi(x)\pi^{(n)}(x)^2.
\end{align*}
\end{linenomath}
The last equality is enough to obtain (\ref{mut-gen}) by conditions (i) and (ii) of Theorem~\ref{thmm:1}. 

In the rest of Step~\ref{s:3}, we handle the voting mechanism behind (\ref{eq:conv2-int}) and show that
\begin{linenomath}
\begin{align}\label{approxgen-1}
\lim_{n\to\infty}\En\left[\,\left| \int_0^t \gamma_n L_{\sf VM}\phi_1\phi_2\circ \mathbf m(\xi_{\gamma_n r})dr -
\int_0^t L_{\sf FV}\phi_1\phi_2(X^{(n)}_r)dr
\right|\,\right]= 0.
\end{align}
\end{linenomath}
Notice that the foregoing equality and (\ref{mut-gen}) thus give (\ref{eq:conv2-int}) by  (\ref{eq:vmgen})  and (\ref{eq:FVgen2}). 
To obtain (\ref{approxgen-1}), we first recall the sequence $\delta_n=s_n'/\gamma_n\to 0$ for $(s_n')$ chosen in Step~\ref{s:2} and use Proposition~\ref{prop:genbdd} to get
\begin{linenomath} 
\begin{align}
&\E^{(n)}\left[\,\left|\int_0^{2\delta_n}\gamma_nL_{\sf VM}\phi_1\phi_2\circ \mathbf m(\xi_{\gamma_n r})dr -
\int_0^{2\delta_n} L_{\sf FV}\phi_1\phi_2(X^{(n)}_r)dr
\right| \, \right]\notag\\
\begin{split}
&\leq C_{\phi_1\phi_2}\Big(\gamma_n\pid^{(n)}+\gamma_n^2\pid^{(n)}\mu_n(\1)\delta_n\Big)\int_0^{2\delta_n}\P^{(n)}(M_{V,V'}>\gamma_n r)dr+C_{\phi_1\phi_2}\delta_n,\label{diff1}
\end{split}
\end{align}
\end{linenomath}
which tends to zero as $n\to\infty$ by the validity of (\ref{eq:exptail}) and condition (ii) of Theorem~\ref{thmm:1}.

With \eqref{diff1} in hand, we complete the proof
  of \eqref{approxgen-1} by proving
\begin{linenomath}
\begin{align}
\begin{split}
&\lim_{n\to\infty}\E^{(n)}\Bigg[\,\Bigg| \int_{2\delta_n}^t \gamma_nL_{\sf VM}\phi_1\phi_2\circ\mathbf m(\xi_{\gamma_n r})dr -
\int_{2\delta_n}^t L_{\sf FV}\phi_1\phi_2(X^{(n)}_r)dr\Bigg|\,\Bigg]=0.\label{diff2}
\end{split}
\end{align}
\end{linenomath}
To get the foregoing limit, we 
will need the estimates:
\begin{linenomath}
\begin{align}
\begin{split}
&\limsup_{n\to\infty}\sup_{\xi\in S^{E_n}}\gamma_n\E^{(n)}_\xi\big[\big|L_{\VM}\phi_1\phi_2 \circ \mathbf m(\xi_{2s_n'})\big|\big]<\infty,\\
&\sup_{n\in \Bbb N}\sup_{\xi\in S^{E_n}}\int_0^t \gamma_n\E^{(n)}_\xi\big[\big|L_{\VM}\phi_1\phi_2\circ \mathbf m(\xi_{\gamma_n r})\big|\big]dr<\infty,\\
&\lim_{n\to\infty}\sup_{\xi\in S^{E_n}}\int_0^{2\delta_n}\gamma_n\E^{(n)}_\xi\big[\big|L_{\VM}\phi_1\phi_2 \circ \mathbf m(\xi_{\gamma_n r})\big|\big]dr=0,
\end{split}\label{group}
\end{align}
\end{linenomath}
which follow from the martingale-difference
argument employed in the proof of
\cite[Theorem~2.2]{CCC}.
The first of these estimates follows from \eqref{eq:keylem} and the
fact the $L_{\sf FV}\phi_1\phi_2$ is bounded, and the last two follow
from Proposition~\ref{prop:genbdd}, (\ref{eq:exptail}),
condition (ii) of Theorem~\ref{thmm:1}, and $\delta_n\to 0$.

Fix $n$ and $\xi\in S^{E_n}$. Now if we define
\begin{linenomath}
\begin{equation*}
H_n(s) = \gamma_n L_{\sf VM}\phi_1\phi_2\circ \mathbf
m(\xi_{\gamma_n s}) - 
\E^{(n)}_\xi\big[\gamma_n L_{\sf VM}\phi_1\phi_2\circ \mathbf
m(\xi_{\gamma_n s})\big| \F_{s-2\delta_n}^{(n)}\big],
\qquad 
s\ge
2\delta_n ,
\end{equation*}
\end{linenomath}
then the expectation in \eqref{diff2} is bounded
  above by
\begin{linenomath}
\begin{align*}
\E^{(n)}_\xi\left[\left( \int_{2\delta_n}^t H_n(s) ds\right)^2\right]^{1/2}
&+
\E^{(n)}_\xi\left[ \int_{2\delta_n}^t\left|
\E^{(n)}_\xi\big[\gamma_n L_{\sf VM}\phi_1\phi_2\circ \mathbf
m(\xi_{\gamma_n s})\big| \F_{s-2\delta_n}^{(n)}\big]
- L_{\sf
  FV}\phi_1\phi_2\big(X^{(n)}_{s-2\delta_n}\big)\right|ds\right]\\
&+ \E^{(n)}_\xi\left[\left| \int_{2\delta_n}^t
\Big(L_{\sf  FV}\phi_1\phi_2\big(X^{(n)}_{s-2\delta_n}\big)
-L_{\sf  FV}\phi_1\phi_2\big(X^{(n)}_{s}\big)\Big)ds\right|\right].
\end{align*}
\end{linenomath}
Each of these terms above tends to 0 for the following reasons. The third term tends
  to 0 because $L_{\sf FV}\phi_1\phi_2$ is bounded and $\delta_n\to
  0$. The second term tend to 0 by the Markov property and
  \eqref{eq:keylem}. The proof that the first term tends to
0 starts with the observation that since
$\E^{(n)}_\xi[H_n(s)H_n(r)]=0$ whenever $s-r>2\delta_n$, we have 
\begin{linenomath}
\begin{equation*}
\E^{(n)}_\xi\left[\left( \int_{2\delta_n}^t H_n(s)ds \right)^2\right]
= 2\int_{2\delta_n\le r\le s\le
  t}\E^{(n)}_\xi\big[H_n(s)H_n(r)\1_{\{r>s-2\delta_n\}} \big]dsdr .
\end{equation*}
\end{linenomath}
The argument showing the last integral tends to 0 is
straightforward given the estimates in (\ref{group}) above, but is somewhat
lengthy. Since it is essentially the same argument
starting at \cite[(6.20)]{CCC}, we refer the reader there
for details.

Hence, (\ref{approxgen-1}) holds by (\ref{diff1}) and (\ref{diff2}). The proof is complete.
\end{step}
\vspace{-.8cm}
\end{proof}

The following lemma is our second step toward the proof of Theorem~\ref{thmm:1} and obtains the necessary $C$-tightness of the empirical measures under consideration for convergence to a Fleming-Viot process.

\begin{lem}\label{lem:CX}
Under the assumptions of Theorem~\ref{thmm:1}, the sequence of laws of $\{X^{(n)}\}$ is $C$-tight as probability measures on $D\big(\R_+,\ms P(S)\big)$. 
\end{lem}
\begin{proof}
In this proof, we show that for $\phi\in \Phi_1$ with $\phi(\lambda)=\langle f,\lambda\rangle$ for $f\in \C(S)$, both the sequences of laws of $\{\phi\big(A^{(n)}\big)\}$ and $\{\phi\big(M^{(n)}\big)\}$ are $C$-tight as probability measures on $D(\R_+,\R)$. The $C$-tightness of these sequences will be enough for the present lemma by Jakubowski's theorem (e.g. \cite[Theorem II.4.1]{P:DW}), since the type space $S$ is assumed to be compact and condition (iii) of Theorem~\ref{thmm:1} is in force.

We make two observations in order to prove the $C$-tightness of the sequence of laws of $\{\phi(A^{(n)})\}$. First, for every $0\leq s<t<\infty$,
\begin{linenomath}
\begin{align}
\big|\phi\big(A^{(n)}_t\big)-\phi\big(A^{(n)}_s\big)\big|
\leq &\gamma_n\int_s^t|L_\VM \phi\circ \mathbf m(\xi_{\gamma_nr})|dr+C_\phi \gamma_n\mu_n(\1)(t-s)\label{Aest}
\end{align}
\end{linenomath}
by the definition (\ref{eq:vmgen}) of $L^{\mu_n}_\VM$. Next,
we observe that (\ref{eq:gen}) gives
\begin{linenomath}
\begin{align}
\begin{split}
\sup_{\xi\in S^{E_n}}\gamma_n\int_0^t\E^{(n)}_\xi[|L_\VM\phi\circ \mathbf m(\xi_{\gamma_nr})|]dr
\leq &C_\phi \gamma_n\pid^{(n)}\int_0^t \P(M_{V,V'}>\gamma_nr)dr\\
&\hspace{-1.5cm}+
C_\phi\gamma_n^2\pid^{(n)}\mu(\1)t\int_0^t \P^{(n)}(M_{V,V'}>\gamma_nr)dr.\label{Aest1}
\end{split}
\end{align}
\end{linenomath}
Then it is readily checked by the Markov property of voter models that for all $K\in \Bbb N$,
\begin{linenomath}
\begin{align}\label{eq:Aldous}
\lim_{\theta \to 0}\limsup_{n\to\infty}\sup_{S,T:S\leq T\leq S+\theta}\E^{(n)}\left[\big|\phi\big(A^{(n)}_T\big)-\phi\big(A^{(n)}_S\big)\big|\right]=0,
\end{align}
\end{linenomath}
where $S,T$ range over all $(\F^{(n)}_{t})$-stopping times bounded by $K$. In more detail, one conditions on $\F^{(n)}_S$ when bounding the expectation in the foregoing display for a fixed pair $(S,T)$ and then handles an expectation for the voter model started at $\xi_S$ by
(\ref{Aest}) and (\ref{Aest1}), which requires conditions (i), (ii) and (iv-1) or (iv-2) of Theorem~\ref{thmm:1} and the use of (\ref{eq:exptail}). Since $\phi$ is a bounded function, (\ref{eq:Aldous}) is enough to fulfill Aldous's condition on $C$-tightness for the sequence of laws of $\{\phi(A^{(n)})\}$ (cf. \cite[Theorem VI.4.5]{JS}). 

To obtain the $C$-tightness of the sequence of laws of $\{\phi(M^{(n)})\}$, 
we first observe that the predictable quadratic variations of $\phi(M^{(n)})$ is given explicitly by
\begin{linenomath}
\begin{align}
\begin{split}\label{mg-FV-qv}
\big\langle \phi\big(M^{(n)}\big)\big\rangle_t =& \gamma_n\int_0^t\Big(
L^{\mu_n}_{\sf VM}\phi^2\circ\mathbf m(\xi_{\gamma_n r})-2\phi\circ\mathbf m(\xi_{\gamma_n r}) L^{\mu_n}_{\sf VM}\phi\circ\mathbf m (\xi_{ \gamma_n r})\Big) dr.
\end{split}
\end{align}
\end{linenomath}
The foregoing equation follows since $\Phi$ is closed under multiplication and $\phi\big(M^{(n)}\big)$ and $\phi^2\big(M^{(n)}\big)$ are both $(\F_t^{(n)})$-martingales
(cf. \cite[Exercise II.29]{EK:MP}). We apply the boundedness of $\phi$ and (\ref{Aest1}) to (\ref{mg-FV-qv}), and argue as in the case of $\phi(A^{(n)})$
that (\ref{eq:Aldous}) holds with $\phi(A^{(n)}_T)$ and $\phi(A^{(n)}_S)$ replaced by $\langle \phi(M^{(n)})\rangle_T$ and $\langle \phi(M^{(n)})\rangle_S$, respectively. By \cite[Theorem VI.4.5]{JS} again, we have the $C$-tightness of the sequence of laws of $\{\langle \phi(M^{(n)})\rangle \}$. Then, since 
$\phi(M^{(n)})$ has jumps bounded by $C_\phi\pi^{(n)}_{\max}\to 0$ by condition (i) of Theorem~\ref{thmm:1}, \cite[Theorem VI.4.13, Proposition VI.3.26]{JS} apply and we obtain the required $C$-tightness of the sequence of laws of $\{\phi(M^{(n)})\}$. The proof is complete.
\end{proof}

We are now ready to prove the first main result of this paper.

\begin{proof}[\bf Proof of Theorem~\ref{thmm:1}.] By Skorokhod's representation (cf. \cite[Theorem~3.1.8]{EK:MP}), we may assume by Lemma~\ref{lem:CX} and a change of probability spaces that $\{X^{(n)}\}$ converges in distribution to some $X$ taking values in $C\big(\R_+,\ms P(S)\big)$. Then for any $\phi\in \Phi_1$ with $\phi(\lambda)=\langle f,\lambda\rangle$ for $f\in \C(S)$, 
we deduce from Lemma~\ref{lem:qvbnd} and (\ref{mg-FV-qv}) that $\phi(X_t)-\phi(X_0)-\int_0^t L_{\sf FV}\phi(X_r)dr$ is a continuous martingale with quadratic variation
\[
\int_0^t \Big(
L^{\mu}_{\sf FV}\phi^2(X_r)-2\phi(X_r) L^{\mu}_{\sf FV}\phi(X_r)\Big) dr.
\]
More precisely, to justify the above form of quadratic variation for $\phi(X)$, we need the fact that
$x\mapsto \int_0^\cdot  L^\mu_\FV \phi(x_r)dr$ is a bounded continuous function from $D\big(\R_+,\ms P(S)\big)$ into $D(\R_+,\R)$ for any $\phi\in \Phi$,
and the uniqueness of quadratic variations for martingales.
Then by integration by parts for continuous semimartingales
 and the preceding martingale characterization of $\phi(X)$ for $\phi\in \Phi_1$, we deduce that $X$ is a continuous $\ms P(S)$-valued process satisfying the same martingale problem for the $\mu$-Fleming-Viot process over $\Phi$ as recalled in (\ref{def:FV}) (or see \cite[p.2]{P:CDW}). This completes
the proof of Theorem~\ref{thmm:1}.
\end{proof}

\section{Weak atomic convergence of  empirical measures for voter models}\label{sec:atom}
In this section, we proceed to the weak atomic convergence of empirical measures for voter models in the limit of a large state space. Let us briefly recall the setup in \cite[Section~2]{EK:AT} for weak atomic convergence of finite measures and its implications. Let $d$ be a metric on the compact type space $S$. 
Fix a continuous function  $J:\R_+\to [0,1]$ such that $J(0)=1$ and $J\equiv 0$ on $[1,\infty)$. Let $\rho$ be the Prohorov metric on the set $\ms M_f(S)$ of finite measures on $S$, and define a metric $\rho_a$ on $\ms M_f(S)$ by
\begin{linenomath}
\begin{align}
\begin{split}
\rho_a(\lambda,\nu)=&\rho(\lambda,\nu)\\
+\sup_{0<\vep\leq 1}&\left|\int_{S^2} J\left(\frac{d(\sigma,\tau)}{\vep}\right)d\lambda^{\otimes 2}(\sigma,\tau)-\int_{S^2} J\left(\frac{d(\sigma,\tau)}{\vep}\right)d\nu^{\otimes 2}(\sigma,\tau)\right|.\label{def:rhoa}
\end{split}
\end{align}
\end{linenomath}
This metric $\rho_a$, finer than $\rho$, still keeps $\ms M_f(S)$ a complete separable space and generates the same Borel sigma-field as $\rho$ does (see \cite[Lemma~2.3 and p.5]{EK:AT}). 

To see the role of the second term on the right-hand side of (\ref{def:rhoa}), notice that
corresponding to every finite measure $\lambda$ on $S$ is an atomic measure for the distributions of the atoms of $\lambda$ defined by
\[
\lambda^*\doteq \sum_{\sigma\in S}\lambda(\{\sigma\})^2\delta_{\sigma}.
\]
With the foregoing definition, the readers may note that the difference between the double-integral terms in (\ref{def:rhoa}) for small enough $\vep>0$ means the approximate difference between $\lambda^*(S)$ and $\nu^*(S)$ through the mollifier $J(\cdot/\vep)$. In fact, it can be shown that $\rho_a(\lambda_n,\lambda)\to 0$ if and only if $\rho(\lambda_n,\lambda)\to 0$ and $\lambda_n^*(S)\to \lambda^*(S)$ (see \cite[Lemma~2.2]{EK:AT}), where the last convergence (for the ``starred-measures'') has the obvious interpretation that there is no loss of atoms in the limit.

Below we study limit theorems of the empirical measures for finite voter models (they are probability measures on $S$), and now the space $\ms P(S)$ is equipped with the metric $\rho_a$ finer than the Prohorov metric $\rho$ underlying Theorem~\ref{thmm:1}. For this purpose and since $(\ms P(S),\rho_a)$ is a closed subspace of $(\ms M_f(S),\rho_a)$, we may and will restrict our application of $\rho_a$ to $\ms P(S)$ from now on.

Our second main result in this paper proves the weak atomic convergence of empirical measures for finite voter models.

\begin{thm}\label{thmm:2}
Let $\{X^{(n)}\}$ be a sequence of empirical measures for $(E_n,q^{(n)},\mu_n)$-voter models time-changed by $\gamma_n$ as in (\ref{timechange}). Assume that the sequence $\{X^{(n)}\}$ converges weakly to a $\mu$-Fleming-Viot process $X$ on $D\big(\R_+,\ms P(S)\big)$ for $\ms P(S)$ equipped with the Prohorov metric.  

Now equip the space $\ms P(S)$ with the metric $\rho_a$ instead. Then $X$ has almost all sample paths in $C\big(\R_+,\ms P(S)\big)$. In addition, if the sequence of laws of $\{X_0^{(n)}\}$ converges to the law of $X_0$ as probability measures in $\ms P(S)$ and we have 
\begin{linenomath}
\begin{align}\label{assump:mu}
\lim_{\vep \searrow 0+}\sup_{\tau\in S}\limsup_{n\to\infty}\;\gamma_n\mu_n\{\sigma\in S;0<d(\sigma,\tau)\leq \vep\}=0,
\end{align}
\end{linenomath}
then the sequence of laws of $\{X^{(n)}\}$ converges as probability measures on $D\big(\R_+,\ms P(S)\big)$ to the law of the $\mu$-Fleming-Viot process $X$.
\end{thm}
\begin{proof}
That $X$ has almost all sample paths in $C\big(\R_+,\ms P(S)\big)$ follows from \cite[Theorem~3.1]{EK:MP}. For the second assertion, we will show that for every $T,\delta>0$, there exists $\vep>0$ such that
\begin{linenomath}
\begin{align}\label{prob-bdd}
\P^{(n)}\left(\sup_{0\leq t\leq T}F_{f_\vep}(\xi_{\gamma_n t})\geq \delta\right)\leq \delta
\end{align}
\end{linenomath}
for all large enough $n$, where $F_{f_\vep}$ is as defined in (\ref{def:Ff}) for
\[
f_\vep(\sigma,\tau)=J\left(\frac{d(\sigma,\tau)}{\vep}\right)-\1_{\{d(\sigma,\tau)=0\}}.
\]
Since 
\[
F_{f_\vep}(\xi_{\gamma_nt})=\int_{S^2}J\left(\frac{d(\sigma,\tau)}{\vep}\right)dX^{(n)\otimes 2}_t(\sigma,\tau)-(X^{(n)}_t)^*(S)
\]
by definition,
the bound (\ref{prob-bdd}) fulfills the condition of \cite[(2.21)]{EK:AT}, and so implies the required weak atomic convergence of the sequence of laws of $\{X^{(n)}\}$ toward $X$. 

The proof of (\ref{prob-bdd}) relies on some facts which we now state. First,
since $f_\vep$ is symmetric and nonnegative, and vanishes on the diagonal, it follows from Proposition~\ref{prop:genF} that
\begin{linenomath}
\begin{align*}
\begin{split}
L^{\mu_n}_\VM F_{f_\vep}(\xi)\leq &2\sum_{x\in E_n}\pi^{(n)}(x)\int_S f_\vep\big(\sigma,\xi(x)\big)d\mu_n(\sigma).
\end{split}
\end{align*}
\end{linenomath}
Hence, we have
\begin{linenomath}
\begin{align}\label{dom:mg}
F_{f_\vep}(\xi_{\gamma_nt})=&F_{f_\vep}(\xi_0)+\int_0^t \gamma_n L^{\mu_n}_{\sf VM}F_{f_\vep}(\xi_{\gamma_n r})dr+M_t\notag\\
\leq &F_{f_\vep}(\xi_0)+\int_0^t2\gamma_n\sum_{x\in E_n}\pi^{(n)}(x)\int_S f_\vep\big(\sigma,\xi_{\gamma_nr}(x)\big)d\mu_n(\sigma)dr+M_t,
\end{align}
\end{linenomath}
where $M$ is a martingale with $M_0=0$, and so the right-hand side of (\ref{dom:mg}) defines a nonnegative submartingale. Second, the map
\begin{linenomath}
\begin{align}\label{map}
\lambda\lmt\left( \int_{S^2}J\left(\frac{d(\sigma,\tau)}{\vep}\right)d\lambda^{\otimes 2}(\sigma,\tau)-\lambda^*(S)\right)\wedge 1
\end{align}
\end{linenomath}
is a bounded continuous function on $\ms P(S)$ by \cite[Theorem 1.2.8]{B:CPM} and \cite[Lemma~2.2]{EK:AT}, and we have
\begin{linenomath}
\begin{align}\label{*conv}
\lim_{\vep\searrow  0+}\left(\int_{S^2}J\left(\frac{d(\sigma,\tau)}{\vep}\right)d\lambda^{\otimes 2}(\sigma,\tau)-\lambda^*(S)\right)=0,\quad \forall\;\lambda\in \ms P(S),
\end{align}
\end{linenomath}
as already discussed before the present theorem in an informal manner.

We are ready to prove (\ref{prob-bdd}). Fix $T,\delta>0$. By the assumed weak atomic convergence of $X^{(n)}_0$ to $X_0$ in distribution, the bounded continuity of the map (\ref{map}) and (\ref{*conv}), 
it holds that for any given $\vep>0$,
\begin{linenomath}
\begin{align}\label{atom:bdd1}
\E^{(n)}[F_{f_{\vep}}(\xi_0)]\leq \frac{\delta^2}{2}
\end{align}
\end{linenomath}
for all large enough $n$. 
In addition,
thanks to (\ref{assump:mu}) and the fact that $f_\vep(\sigma,\tau)$ is supported on $\{(\sigma,\tau)\in S\times S;0<d(\sigma,\tau)\leq \vep\}$ with $|f_\vep|\leq 1$, we can choose $\vep>0$ such that
\begin{linenomath}
\begin{align}
\sup_{\xi\in S^{E_n}}\gamma_n\sum_{x\in E_n}\pi^{(n)}(x)\int_S f_\vep\big(\sigma,\xi(x)\big)d\mu_n(\sigma)\leq \frac{\delta^2}{4T}\label{atom:bdd2}
\end{align}
\end{linenomath}
for all large $n$. Now by (\ref{dom:mg}) and an application of Doob's weak $L^1$-inequality to the nonnegative submartingale on the right-hand side of (\ref{dom:mg}), we deduce from (\ref{atom:bdd1}) and (\ref{atom:bdd2}) that with respect to the $\vep>0$ chosen for (\ref{atom:bdd2}),
\begin{linenomath}
\begin{align*}
& \P^{(n)}\left(\sup_{0\leq t\leq T}F_{f_\vep}(\xi_{\gamma_n t})\geq \delta\right)\\
\leq &\frac{1}{\delta}\left(\E^{(n)}[F_{f_\vep}(\xi_0)]+\E^{(n)}\left[\int_0^T2\gamma_n\sum_{x\in E_n}\pi^{(n)}(x)\int_S f_\vep\big(\sigma,\xi_{\gamma_ns}(x)\big)d\mu_n(\sigma)ds\right]\right)\leq \delta,
\end{align*}
\end{linenomath}
for all large enough $n$. 
The last inequality proves (\ref{prob-bdd}), and the proof is complete.
\end{proof}

As an application of Theorem~\ref{thmm:2}, we consider the convergence of atom-size point processes for voter models when mutation is absent.
We recall the coalescing Markov chains $(X^x)$,
  and in the following write $N_n=\#E_n$, 
\[
\widehat{\mathsf C}^{(n)}_j=\inf\{t\geq 0;|\{X^x_t;x\in E_n\}|=j\},\quad 1\leq j\leq N_n,
\]
and  $Z_1,Z_2,\cdots$ for a sequence of independent exponential variables with $\E[Z_j]=1/{j\choose 2}$.

\begin{thm}\label{thmm:3}
Equip $\ms P(S)$ with the metric $\rho_a$. 
Let $\{X^{(n)}\}$ be a sequence of empirical measures for $(E_n,q^{(n)},0)$-voter models time-changed by $\gamma_n$ as in (\ref{timechange}), which converges weakly to a Fleming-Viot process $X$ without mutation.

Suppose also that for each $n\in \Bbb N$, the types under the initial condition of the $(E_n,q^{(n)},0)$-voter model are all distinct almost surely and it holds that
\begin{linenomath}
\begin{align}\label{conv:time}
\frac{\widehat{\mathsf C}^{(n)}_j}{\gamma_n}\convdn \sum_{i=j+1}^\infty Z_i,\quad \forall\;j\in \Bbb N.
\end{align}
\end{linenomath}
Then for every continuous function $f:[0,1]\to  \R$, the process
\begin{linenomath}
\begin{align}\label{fseries}
\sum_{i=1}^\infty f\big(\mathfrak a_i(X_s)\big),\quad s\in (0,\infty),
\end{align}
\end{linenomath}
is continuous, and for every fixed $t\in (0,\infty)$,
\[
\left(\sum_{i=1}^\infty f\big(\mathfrak a_i(X_s^{(n)})\big);\; s\geq t\right)\convdn\left(\sum_{i=1}^\infty f\big(\mathfrak a_i(X_s)\big);\; s\geq t\right)
\]
as probability measures on $D\big([t,\infty),\R\big)$.     
\end{thm}

The convergence in (\ref{conv:time}) under appropriate
conditions is proven in\cite[Theorem~1.1, Theorem~
1.2]{Oliveira_2013}, resolving an open problem in
\cite{aldous-fill-2014}.  
These conditions are enough (applying
Theorems~\ref{thmm:1} and \ref{thmm:2}) for the weak convergence of $\{X^{(n)}\}$ to the Fleming-Viot process without mutation as required in Theorem~\ref{thmm:3}.
See also \cite[Section~2]{CCC} for related results in terms of the Wright-Fisher diffusion.

\begin{rmk}
(1) Implicit in the conclusion of Theorem~\ref{thmm:3} is the fact that (\ref{fseries}) is a  finite sum for all $s>0$ almost surely.\\

\noindent (2) As particular applications of Theorem~\ref{thmm:3}, we obtain the weak convergence of the  entropy processes and diversity processes associated with the voter models under consideration (cf. the definitions in (\ref{def:ent}) and (\ref{def:div})). \qed 
\end{rmk}

\begin{proof}[Proof of Theorem~\ref{thmm:3}]
For $\lambda\in \ms P(S)$ with masses of its atoms arranged in the decreasing order $\mathfrak a_1(\lambda)\geq \mathfrak a_2(\lambda)\geq \cdots 
$, we define
\[
\mathcal  S^f_L(\lambda)\doteq\sum_{i=1}^{L\wedge \mathfrak N(\lambda)}f\big(\mathfrak a_i(\lambda)\big),\quad L\in \Bbb N\cup \{\infty\},
\]
where $\mathfrak N(\lambda)\in \Bbb Z_+\cup \{\infty\}$ is the number of atoms of $\lambda$. In the following, we may assume that $X^{(n)}\to X$ a.s. in $D\big(\R_+,\ms P(S)\big)$ by a change of probability spaces and Skorokhod's representation (cf. \cite[Theorem~3.1.8]{EK:MP}). Since $X$ takes values in $C\big(\R_+,\mathscr M_f(S)\big)$ by Theorem~\ref{thmm:2} where $\ms P(S)$ is equipped with $\rho_a$, it follows from \cite[Theorem III.10.1]{EK:MP} that almost surely, 
\begin{linenomath}
\begin{align}\label{rhoa-conv}
\rho_a(X^{(n)}_t,X_t)\xrightarrow[n\to\infty]{}0\quad \forall \;t\in \R_+.
\end{align}
\end{linenomath}

\begin{claim}\label{c:1}
We have 
\begin{linenomath}
\begin{align}\label{ineq:at1}
\lim_{L\to\infty}\Big(\P(\mathfrak N(X_t)>L)+\limsup_{n\to\infty}\P(\mathfrak N(X^{(n)}_t)>L)\Big)=0.
\end{align}
\end{linenomath}

To see (\ref{ineq:at1}), we note that for all $L\in \Bbb N$,
\begin{linenomath}
\begin{align*}
\P(\mathfrak N(X_t)>L)\leq \P(\mathfrak a_{L+1}(X_t)>0)
\leq \liminf_{n\to\infty}\P^{(n)}\left(\mathfrak a_{L+1}(X^{(n)}_t)>0\right),
\end{align*}
\end{linenomath}
where the last inequality follows from (\ref{rhoa-conv}) and \cite[Lemma 2.5]{EK:AT}. We 
observe that $\mathfrak a_{L+1}(X^{(n)}_t)>0$ implies that the number of distinct types under $\xi_{\gamma_n t}$, that is $\#\{\xi_{\gamma_n t}(x);x\in E_n\}$ (with $\{\xi_{\gamma_nt}(x);x\in E_n\}$ regarded as a set), is at least $L+1$ under $\P^{(n)}$. 
Also, we observe from duality (Section~\ref{sec:dual}) that $\big(\xi_{\gamma_nt}(x);x\in E_n\big)$ and $\big(\xi_0(X_{\gamma_nt}^x);x\in E_n\big)$ have the same distribution (see (\ref{prob:dual})). Since all of the coordinates of $\xi_0$ under $\P^{(n)}$ are distinct by assumption, applying these two observations to the last inequality gives
\begin{linenomath}
\begin{align*}
\lim_{L\to\infty}\P(\mathfrak N(X_t)>L)\leq &\lim_{L\to\infty}\lim_{n\to\infty}\P\left(\widehat{\mathsf C}^{(n)}_{L+1}\geq \gamma_nt\right)
=\lim_{L\to\infty}\P\left(\sum_{j=L+2}^\infty Z_j\geq t\right)=0,
\end{align*}
\end{linenomath}
where the first equality follows from (\ref{conv:time}). 
Similarly, 
\begin{linenomath}
\begin{align*}
\lim_{L\to\infty}\limsup_{n\to\infty}\P\big(\mathfrak N(X^{(n)}_t)>L\big)=0.
\end{align*}
\end{linenomath}
The last two displays prove our claim (\ref{ineq:at1}). 
\end{claim}

\begin{claim}\label{c:2}
We have (1) 
\begin{linenomath}
\begin{align}\label{atom:conv}
\big(\mathfrak a_1(X^{(n)}),\mathfrak a_2(X^{(n)}),\cdots\big)\xrightarrow[n\to\infty]{\mbox{a.s.}}\big(\mathfrak a_1(X),\mathfrak a_2(X^{(n)}),\cdots\big).
\end{align}
\end{linenomath}
in $D(\R_+,\R^{\Bbb N}_+)$, where $\R_+^{\Bbb N}$ is endowed with the metric $\big((x_n),(y_n)\big)\mapsto \sum_{n=1}^\infty \frac{|x_n-y_n|\wedge 1}{2^n}$, and (2)
each $\mathfrak a_j(X)$ takes values in $C(\R_+,\R_+)$ almost surely.

To see (1), we notice that $\lambda \mapsto (\mathfrak a_1(\lambda),\mathfrak a_2(\lambda),\cdots):\mathscr M_f(S)\to \R_+^{\Bbb N}$ is continuous by \cite[Lemma~2.5]{EK:AT}. 
Hence, by \cite[Proposition~3.6.5 or Exercise 3.13]{EK:MP}, the almost-sure convergence of $X^{(n)}$ to $X$ in $D\big(\R_+,\ms P(S)\big)$ implies (\ref{atom:conv}). For (2), the continuity of $\mathfrak a_j(X)$ follows from \cite[Lemma~2.5]{EK:AT} and the fact that $X$ takes values in $C\big(\R_+,\ms P(S)\big)$ almost surely (see the beginning of the present proof). 
\end{claim}

We are ready to prove Theorem~\ref{thmm:3}. First we note that by the continuity of $s\mapsto \mathfrak a_j(X_s)$ in Claim~\ref{c:2}, $s\mapsto \mathcal S_{L'}^f(X_s)$ is continuous on $[t,\infty)$ for any $L'\in \Bbb N\cup \{\infty\}$, which proves the first assertion of the theorem in particular. For the second assertion, we consider the following inequality: for every $t,T,\delta,L>0$ with $t<T$,
\begin{linenomath}
\begin{align}
&\P\left(\sup_{s\in [t,T]}\left|\mathcal S^f_\infty(X^{(n)}_s)-\mathcal S^f_\infty(X_s)\right|>\delta\right)\notag\\
\begin{split}\label{ineq:at}
\leq
&\P\left(\sup_{s\in [t,T]}|\mathcal S_L^f(X^{(n)}_s)-\mathcal S^f_L(X_s)|>\delta\right)
 +\P(\mathfrak N(X^{(n)}_t)>L)\\
 &+\P\left(\mathfrak N(X_t)>L\right),
\end{split}
\end{align}
\end{linenomath}
which follows from the fact that $s\mapsto \mathfrak N(X^{(n)}_s)$ and $s\mapsto \mathfrak N(X^{(n)}_s)$ are finite and decreasing on $[t,\infty)$ almost surely by Claim~\ref{c:1}, (\ref{atom:conv}) in Claim~\ref{c:2} and the assumption that mutation is absent. To handle the first term on the right-hand side of (\ref{ineq:at}), we use both (1) and (2) in Claim~\ref{c:2}
and standard properties of convergence of c\`adl\`ag functions to continuous functions (cf. \cite[Section 3.10]{EK:MP}) to get
\[
\limsup_{n\to\infty}\P\left(\sup_{s\in [t,T]}|\mathcal S_L^f(X^{(n)}_s)-\mathcal S_L^f(X_s)|>\delta\right)=0.
\]
Claim~\ref{c:1} is able to handle the other two terms on the right-hand side of (\ref{ineq:at}). Hence, by (\ref{ineq:at}), we obtain
\[
\lim_{n\to\infty}\P\left(\sup_{s\in [t,T]}\left|\mathcal S^f_\infty(X^{(n)}_s)-\mathcal S_\infty^f(X_s)\right|>\delta\right)=0,
\]
which is enough to obtain the required convergence (cf. \cite[Section 3.10]{EK:MP}). We have proved our assertion for the atom-size point processes for the voter models under consideration. The proof is complete.
\end{proof}

\section{Duality}\label{sec:dual}
In this section, we discuss the duality between multi-type voter models with mutation and coalescing Markov chains and prove Proposition~\ref{prop:dual}. A similar treatment of the duality can be found in, for example,  \cite{GM:NVM} for noisy voter models.

\subsection{Graphical representation}
First let $S^x_k$ and $\{U^x_1<U^x_2<\cdots\}$, for $x\in E$ and $k\geq 1$, be independent such that $S^x_k$ are $E$-valued with distribution $\P(S^x_k=y)\equiv q(x,y)$, and $U^x_k$'s are the arrival times of a rate-$1$ Poisson process;
these $S^x$ and $U^x$ will be used to describe the $(E,q,0)$-voter model where mutation is absent. To incorporate mutation with respect to  $\mu$ satisfying $\mu(\1)>0$, we let $M^x_k$ and $\{V^x_1<V^x_2<\cdots\}$, for $x\in E$ and $k\geq 1$, be independent such that $M^x_k$ are i.i.d. $S$-valued with law $\overline{\mu}$, where $\overline{\mu}=\mu/\mu(\1)$ as before, and $V^x_k$'s are the arrival times of a rate-$\mu(\1)$-Poisson process. Moreover, $S^x,U^x,M^x,V^x$ are independent.

Now, given an initial condition $\xi\in S^E$, we have a version of a $(E,q,\mu)$-voter model which is defined as a pure-jump process with updating times $U_k^x$'s and $V_k^x$'s. At the times $U_k^x$, we draw an arrow at the space-time point $(x,U^x_k)$ pointing to the space-time point $(S^x_k,U^x_k)$, which means informally that site $x$ adopts the type at site $S^x_k$, and set $\xi_t(x)=\xi_{t-}(S^x_k)$. At the mutation times $t=V^x_k$, we set $\xi_t(x)=M^x_k$.

To establish the duality, at any fixed time $t>0$, first we reverse time and identify a family of $q$-coalescing Markov chains $\{X^{x,t};x\in E\}$ which keep track of the genealogy of type propagation of the $(E,q,0)$-voter model without mutation.
We set $X^{x,t}_0\equiv x$ and then let $X^{x,t}_s$ trace out a path going backward in time down to the starting time of the voter model, following the arrows defined in the previous paragraph. Here and below, ``backward'' is in
terms of the time progression of the voter model. More precisely, if $\{U^x_k;k\in \Bbb N\}\cap (0,t)=\varnothing$, we put $X^{x,t}_s=x$ for $s\in [0,t]$; otherwise, we single out the last update defined by $(S^x,U^x)$ for the voter model without mutation by choosing $k_0=\max\{k\geq 1;U^x_k<t\}$, and set
\[
X^{x,t}_s=x\quad\mbox{for $s\in \big(0,t-U^x_{k_0}\big)$}\quad  \mbox{and}\quad\mbox{$X^{x,t}_{t-U^x_{k_0}}=S^x_{k_0}$}.
\]
We repeat this construction starting at $\big(X^{x,t}_{t-U^x_{k_0}},t-U^x_{k_0}\big)$ going backward in time, thus defining $X^{x,t}_s$ for  $s$ up to $t$ by an induction on the time intervals $[U^x_{k_0-1},U^x_{k_0}]$, ..., $[0,U^x_1]$. It should be plain that for fixed $t$, $X^{x,t}$ are $q$-Markov chains by the reversibility of Poisson processes, and their coalescence after pairwise meeting follows from the use of the arrows.

The duality between the $(E,q,\mu)$-voter model and the $q$-coalescing Markov chains $\{X^{x,t};x\in E,t\geq 0\}$ can be described as follows. We introduce $\Pi_t=\{(y,V^y_k);y\in E,V^y_k\leq t\}$, which consists of all space-time points where mutation events occur up to time $t$. Then we consider mutation events in terms of the type propagation genealogy of the $(E,q,0)$-voter model {\it without} mutation up to time $t$, or equivalently, in terms of the union of space-time trajectories of $X^{x,t}$ for $x\in E$. 
If the chain $X^{x,t}$ does not encounter a mutation event in $\Pi_t$ throughout its trajectory in the sense that $(X^{x,t}_s,t-s)\notin \Pi_t$ for all $s\in [0,t]$, we put $e(x,t)=\infty$. Otherwise, we consider the first mutation event on the trajectory of $X^{x,t}$. It corresponds to the last mutation event along the unique space-time ``ancestral line'' of type propagation under the voter-model dynamics without mutation, which leads to the type at site $x$ and time $t$, and finalizes the type being transported to the destination $(x,t)$. Therefore we choose $y $ and $k$ satisfying $X^{x,t}_{t-V^y_k}=y$ and $(X^{x,t}_s,t-s)\notin \Pi_t$ for $s<t-V^y_k$. Set $e(x,t)=t-V^y_k$ and $M(x,t)=M^y_k$ which give the time spent by $X^{x,t}$ before finding that mutation event and the type of the associated mutant, respectively. Then we see that the $(E,q,\mu)$-voter model with initial condition $\xi$ defined above by $S^x,U^x,M^x,V^x$ satisfies the equation
\begin{linenomath}
\begin{align}\label{prob:dual}
\xi_t(x)=M(x,t)\1_{\{e(x,t)\leq t\}}+\xi\big(X^{x,t}_t\big)\1_{\{e(x,t)>t\}}\quad\forall\;x\in E
\end{align}
\end{linenomath}
almost surely for each fixed $t$, which gives the required duality.

\subsection{Proof of Proposition~\ref{prop:dual}}
To see how (\ref{f:bdd}) follows from (\ref{prob:dual}) , we work with the partition $\{A_j\}_{1\leq j\leq 4}$ defined by 
 \begin{linenomath}
\begin{align*}
A_1=&\{e(x,t)\wedge e(y,t)>t\}, \\
A_2=&\{M_{x,y}<e(x,t)\wedge e(y,t)\leq t\},\\
A_3=&\{e(x,t)\wedge e(y,t)\leq M_{x,y}\leq t\},\\ A_4=&\{e(x,t)\wedge e(y,t)\leq t<M_{x,y}\}. 
\end{align*}
 \end{linenomath}
Then we consider the differences 
\begin{linenomath}
\begin{align}\label{diff:f}
&\E\left[f\big(\xi_t(x),\xi_t(y)\big);A_j\right]-\E\left[f\big(\xi(X^{x,t}_t),\xi(X^{y,t}_t)\big);A_j\right]
\end{align}
\end{linenomath}
for $1\leq j\leq 4$. For $j=1$, there is no mutation throughout the trajectories of $X^{x,t}$ and $X^{y,t}$, and so $\xi_t(x)=\xi(X^{x,t}_t)$ and $\xi_t(y)=\xi(X^{y,t}_t)$.  For $j=2$, the two terms in the above display are both zero,
since the two chains $X^{x,t}$ and $X^{y,t}$ coalesce before the first mutation events on their trajectories, which are the same as a result. In other words, for $j=1,2$, the above difference is zero.
For $j=3$, we write $\mathbf e_1$ and $\mathbf e_2$ for two independent exponential variables with mean one, and obtain 
\begin{linenomath}
\begin{align*}
&\left|\E\left[f\big(\xi_t(x),\xi_t(y)\big);A_3\right]-\E\left[f\big(\xi(X^{x,t}_t),\xi(X^{y,t}_t)\big);A_3\right]\right|\\
&\hspace{1cm}\leq \P(e(x,t)\wedge e(y,t)<M_{x,y}<t)\\
&\hspace{1cm}\leq \int_0^t \P(\mathbf e_1\wedge \mathbf e_2\leq s)\P(M_{x,y}\in ds)
\leq 2\mu(\1)\int_0^t \P(M_{x,y}>s)ds,
\end{align*}
\end{linenomath}
where the second inequality follows from the independence $\{X^{z,t};z\in E\}\ind \{V^z;z\in E\}$ and the reversibility of Poisson processes. For $j=4$, the same reason applies and we get
\[
\left|\E\left[f\big(\xi_t(x),\xi_t(y)\big);A_4\right]-\E\left[f\big(\xi(X^{x,t}_t),\xi(X^{y,t}_t)\big);A_4\right]\right|\leq (1-e^{-2\mu(\1)t})\P(M_{x,y}>t).
\]

The required inequality (\ref{f:bdd}) follows by putting together these observations for the terms in (\ref{diff:f}) for $1\leq j\leq 4$.


\begin{thebibliography}{9}
\bibitem{Aldous_1982} 
D. Aldous, Markov chains with almost exponential hitting times, 
 Stochastic Process. Appl. 13 (1982) 305--310. \href{http://dx.doi.org/10.1016/0304-4149(82)90016-3}{DOI:10.1016/0304-4149(82)90016-3}


\bibitem{Aldous_1991}
D. Aldous, Meeting times for independent {M}arkov chains, Stochastic Process. Appl. 38 (1991) 185--193. \href{http://dx.doi.org/10.1016/0304-4149(91)90090-y}{DOI:10.1016/0304-4149(91)90090-y}

\bibitem{Aldous_FMIE}
D. Aldous, Interacting particle systems as stochastic social dynamics, Bernoulli 19 (2013) 1122--1149. \href{http://dx.doi.org/10.3150/12-bejsp04}{DOI:10.3150/12-bejsp04}



\bibitem{aldous-fill-2014}
D. Aldous, J.A. Fill, Reversible {M}arkov {C}hains and {R}andom {W}alks on {G}raphs, Unfinished monograph, 2002, available online: \url{http://www.stat.berkeley.edu/$\sim$aldous/RWG/book.html}.

 
\bibitem{Biggs_1974}
N. Biggs, Algebraic Graph Theory, second ed., Cambridge University Press, Cambridge, 1993.
\href{http://dx.doi.org/10.1017/cbo9780511608704}{DOI:10.1017/cbo9780511608704}



\bibitem{B:CPM}
P. Billingsley, Convergence of {P}robability {M}easures, second ed.,  John Wiley \& Sons, Inc., New York, 1999.
\href{http://www.ams.org.ezp-prod1.hul.harvard.edu/mathscinet-getitem?mr=1700749}{MR1700749}


\bibitem{Birkner_2009}
M. Birkner, J. Blath, M. M\"ohle, M. Steinr\"ucken, J. Tams, A modified lookdown construction for the {X}i-{F}leming-{V}iot process with mutation and populations with recurrent bottlenecks, ALEA 6 (2009) 25--61. \href{http://www.ams.org/mathscinet-getitem?mr=2485878}{MR2485878}


\bibitem{Bramson_1980}
M. Bramson, D. Griffeath, On the {W}illiams-{B}jerknes tumour growth model: {II}, Math. Proc. Camb. Phil. Soc. 88 (1980) 339--357. \href{http://dx.doi.org/10.1017/s0305004100057650}{DOI:10.1017/s0305004100057650}

\bibitem{Bramson_1981}
M. Bramson, D. Griffeath, On the {W}illiams-{B}jerknes tumour growth model {I}, Ann. Probab. 9 (1981) 173--185. \href{http://dx.doi.org/10.1214/aop/1176994459}{DOI:10.1214/aop/1176994459}


\bibitem{Chen_2013}
Y.-T. Chen, Sharp benefit-to-cost rules for the evolution of cooperation on regular graphs, Ann. Appl. Probab. 23 (2013) 637--664. \href{http://dx.doi.org/10.1214/12-aap849}{DOI:10.1214/12-aap849}




\bibitem{CCC}
Y.-T. Chen, J. Choi, J.T. Cox, On the convergence of densities of finite voter models to the {W}right-{F}isher diffusion,
Ann. Inst. Henri Poincar\'e Probab. Stat. 52 (2016) 286--322.
\href{http://dx.doi.org/10.1214/14-aihp639}{DOI:10.1214/14-aihp639}


\bibitem{Cox_1989}
J.T. Cox, Coalescing random walks and voter model consensus times on the torus in {$\Bbb Z^d$}, Ann. Probab. 17 (1989) 1333--1366. \href{http://dx.doi.org/10.1214/aop/1176991158}{DOI:10.1214/aop/1176991158}





\bibitem{CDP}
J.T. Cox, R. Durrett, E.A. Perkins, Rescaled voter models converge to super-{B}rownian motion, Ann. Probab. 28 (2000) 185--234. \href{http://dx.doi.org/10.1214/aop/1019160117}{DOI:10.1214/aop/1019160117}





\bibitem{CDP_13}
J.T. Cox, R. Durrett, E.A. Perkins, Voter model perturbations and reaction diffusion equations, Ast\'erisque 349 (2013). \href{http://www.ams.org/mathscinet-getitem?mr=3075759}{MR3075759} 


\bibitem{CG}
J.T. Cox, D. Griffeath, Mean field asymptotics for the planar stepping stone model, Proc. London Math. Soc. 61 (1990) 189--208. \href{http://plms.oxfordjournals.org/content/s3-61/1/189.abstract}{DOI:10.1112/plms/s3-61.1.189}



\bibitem{Diaconis/Saloff Coste}
P. Diaconis, L. Saloff-Coste, Logarithmic Sobolev inequalities for finite Markov chains, Ann. Appl. Probab. 6, (1996) 695--750. \href{http://dx.doi.org/10.1214/aoap/1034968224}{DOI:10.1214/aoap/1034968224}



\bibitem{Donnelly_1996}
P. Donnelly, T.G. Kurtz, A countable representation of the {F}leming-{V}iot measure-valued diffusion, Ann. Probab. 24 (1996) 698--742. \href{http://dx.doi.org/10.1214/aop/1039639359}{DOI:10.1214/aop/1039639359}




\bibitem{Donnelly_1999}
P. Donnelly, T.G. Kurtz, Particle representations for measure-valued population models, Ann. Probab. 27 (1999) 166--205. \href{http://dx.doi.org/10.1214/aop/1022677258}{DOI:10.1214/aop/1022677258}








\bibitem{EK:MP}
S.N. Ethier, T.G. Kurtz, Markov Processes. Characterization and Convergence, John Wiley \& Sons, Inc., New York, 1986.
\href{http://www.ams.org/mathscinet-getitem?mr=838085}{MR0838085}




\bibitem{EK:93}
S.N. Ethier, T.G. Kurtz, Fleming-{V}iot processes in population genetics, SIAM J. Control Optim. 31 (1993) 345--386. \href{http://epubs.siam.org/doi/abs/10.1137/0331019}{DOI:10.1137/0331019}




\bibitem{EK:AT}
S.N. Ethier, T.G. Kurtz, Convergence to {F}leming-{V}iot processes in the weak atomic topology, Stochastic Process. Appl. 54 (1994) 1--27. \href{http://www.sciencedirect.com/science/article/pii/0304414994000069}{DOI:10.1016/0304-4149(94)00006-9}







\bibitem{Fleming_1979}
W.H. Fleming, M. Viot, Some measure-valued {M}arkov processes in population genetics theory, Indiana Univ. Math. J. 28 (1979) 817--843. \href{http://www.iumj.indiana.edu/IUMJ/FULLTEXT/1979/28/28058}{DOI:10.1512/iumj.1979.28.28058}



\bibitem{GM:NVM}
B.L. Granovsky, N. Madras, The noisy voter model, Stochastic Process. Appl. 55 (1995) 23--43. \href{http://www.sciencedirect.com/science/article/pii/030441499400035R}{DOI:10.1016/0304-4149(94)00035-R}




\bibitem{JS}
J. Jacod, A.N. Shiryaev, Limit {T}heorems for {S}tochastic {P}rocesses, second ed., Springer-Verlag, Berlin, 2003. \href{http://link.springer.com/book/10.1007/978-3-662-05265-5}{DOI:10.1007/978-3-662-05265-5} 


\bibitem{Keilson_1979}
J. Keilson, Markov Chain Models {\textemdash} {R}arity and {E}xponentiality, Springer, New York, 1979. \href{http://dx.doi.org/10.1007/978-1-4612-6200-8}{DOI:10.1007/978-1-4612-6200-8}

\bibitem{Kingman_1982}
J.F.C. Kingman, The coalescent, Stochastic Process.  Appl. 13 (1982), 235--248. \href{http://dx.doi.org/10.1016/0304-4149(82)90011-4}{DOI:10.1016/0304-4149(82)90011-4}


\bibitem{Levin_2008}
D.A. Levin, Y. Peres, E.L. Wilmer, Markov Chains and Mixing Times, American Mathematical Society, Rhode Island, 2008. \href{http://dx.doi.org/10.1090/mbk/058}{DOI:10.1090/mbk/058}






\bibitem{L:IPS}
T.M. Liggett, Interacting Particle Systems, 
Reprint of the 1985 original, Springer-Verlag, Berlin, 2005.
\href{http://link.springer.com/book/10.1007/b138374}{DOI:10.1007/b138374
}





\bibitem{Moran_1958}
P.A.P. Moran, Random processes in genetics, Math. Proc. Camb. Phil. Soc. 54 (1958) 60--71. \href{http://dx.doi.org/10.1017/s0305004100033193}{DOI:10.1017/s0305004100033193}



\bibitem{MT}
C. Mueller, R. Tribe, Stochastic p.d.e.'s arising from the long range contact and long range voter processes, Probab. Theory Related Fields 102 (1995) 519--545. \href{http://link.springer.com/article/10.1007/BF01198848}{DOI:10.1007/BF01198848}




\bibitem{Oliveira_2012}
R.I. Oliveira, On the coalescence time of reversible random walks, Trans. Amer. Math. Soc. 364 (2012) 2109--2128. \href{http://dx.doi.org/10.1090/s0002-9947-2011-05523-6}{DOI:10.1090/s0002-9947-2011-05523-6}

\bibitem{Oliveira_2013}
R.I. Oliveira, Mean field conditions for coalescing random walks, 
Ann. Probab. 41 (2013)  3420--3461. \href{http://dx.doi.org/10.1214/12-aop813}{DOI:10.1214/12-aop813}


\bibitem{P:CDW}
E.A. Perkins, Conditional
  {D}awson-{W}atanabe processes and {F}leming-{V}iot 
processes, in: Seminar on Stochastic Processes, 1991, in: Progr. Probab., vol. 29, 1992, pp. 143--156. \href{http://dx.doi.org/10.1007/978-1-4612-0381-0\_12}{DOI:10.1007/978-1-4612-0381-0\_12}

\bibitem{P:DW}
E.A. Perkins, {D}awson-{W}atanabe superprocesses and measure-valued diffusions, in: Lectures on {P}robability {T}heory and {S}tatistics (Saint-Flour, 1999), in: Lecture Notes in Math., vol. 1781, Springer-Verlag, 2002, pp.125--324. \href{http://link.springer.com/book/10.1007/b93152}{DOI:10.1007/b93152}.

\bibitem{Shiga_90}
T. Shiga, A stochastic equation based on a Poisson system for a class of measure-valued diffusion processes, J. Math. Kyoto Univ. 30 (1990) 245--279. \href{http://www.ams.org/mathscinet-getitem?mr=MR1068791}{MR1068791}


\end{thebibliography}
\end{document}